\documentclass[12pt,reqno]{amsart}
\usepackage{latexsym,amscd, amssymb,amsmath, amsthm}
 
\usepackage[T2A]{fontenc}
\usepackage[utf8]{inputenc} 
\usepackage[english]{babel}
\newtheorem{theorem}{Theorem}[section]
\newtheorem{example}{Example}[section]
\usepackage{mathrsfs}
\usepackage{mathtools}
\usepackage[vcentermath]{youngtab}
\usepackage{enumitem}

\begin{document}

 \title{ \MakeUppercase{  The Graph Algebra  I: Representation-Theoretic Structure
}} 

\author{Leonid Bedratyuk}
\address{ Khmelnytsky National University, Faculty of Information Technology, Ukraine}
\email{leonidbedratyuk@khmnu.edu.ua}

\begin{abstract}
The paper studies the graph algebra whose monomial basis is naturally
indexed by simple graphs on a fixed set of vertices. This algebra is at the
same time the algebra of pseudo-Boolean functions on the Boolean cube and a
natural object of algebraic combinatorics, related to the Boolean lattice of
subsets of the edge set of the complete graph.

The main aim of the paper is to study two compatible representation-theoretic
structures on this algebra: the action of the Lie algebra $\mathfrak{sl}_2$,
arising from the operators of adding and deleting one edge, and the action of
the pair group $S_n^{(2)}$, induced by the renumbering of vertices. It is
proved that the graph algebra with this $\mathfrak{sl}_2$-action is
isomorphic to a tensor power of the standard two-dimensional
$\mathfrak{sl}_2$-module, and on this basis its decomposition into
irreducible $\mathfrak{sl}_2$-modules is obtained. Primitive spaces, that is,
the kernels of the edge-deletion operator on rank components, are also
described, and it is shown that they have a natural interpretation in terms
of two-row Specht modules.

It is then established that the $\mathfrak{sl}_2$-action commutes with the
action of the pair group. It follows that the space of graph invariants also
inherits the structure of an $\mathfrak{sl}_2$-module. Using Schur--Weyl
duality, primitive invariants are described through the fixed parts of the
restrictions of two-row Specht modules from the full symmetric group on the
edge set to the pair group. As a consequence, the classical enumeration of
non-isomorphic graphs by the number of edges receives a representation-theoretic
refinement: the orbital components entering the Burnside--Polya formula
decompose into natural primitive contributions associated with the
$\mathfrak{sl}_2$-structure and two-row Specht modules.

\end{abstract}

\maketitle 

\section{Introduction}

Let $G$ be a simple graph on a fixed set of $n$ vertices. Since every such
graph is determined by the choice of a subset of the edge set of the complete
graph $K_n$, the set of all simple graphs on these vertices is naturally
identified with the Boolean cube $\{0,1\}^m$,
$
m=\binom n2.
$
Denoting the corresponding edge coordinates by
$x_1,x_2,\ldots,x_m$, we pass from graphs themselves to the algebra of
pseudo-Boolean functions on the Boolean cube, that is, to the quotient algebra
$$
\mathcal A
=
\mathbb C[x_1,\dots,x_m]/(x_1^2-x_1,\dots,x_m^2-x_m),
$$
which in this paper will be called the \emph{graph algebra} and which is the
main object of our study. Thus, the algebra $\mathcal A$ combines two
viewpoints: the combinatorial one, in which its basis is indexed by graphs,
and the functional one, in which its elements are functions on the set of all
graphs. Its basis consists of squarefree monomials
$$
\boldsymbol{x}_S=\prod_{e\in S}x_e,\qquad S\subseteq [m],
$$
which naturally correspond to simple graphs with edge set $S$. The number of
edges defines a natural rank decomposition of the vector space
$$
\mathcal A
=
\bigoplus_{k=0}^{m} A_k,
$$
where
$$
A_k
=
\operatorname{span}_{\mathbb C}
\left\{
\boldsymbol{x}_S:\ |S|=k
\right\}.
$$

Algebras of this type have already appeared in works related to Mobius
algebras, orbit algebras, and reconstruction problems in graph theory
\cite{Mnukhin1992,MnukhinMobius}. A close approach was also developed by
Buchwalder, where the Boolean algebra of functions on a set of subsets is used
to describe generalized orbit algebras \cite{Buch,Buchwalder2010}. It is in
such an orbital setting that averaging operators, orbit sums, subgraph
inclusion coefficients, and transition matrices between orbital layers
naturally arise; these were used in the works of Mnukhin, Pouzet, and Thiery
in connection with reconstruction questions \cite{Mnukhin1992,PT}.

The aim of our work is to study the graph algebra as an object of
representation theory. On the one hand, the algebra $\mathcal A$ has a natural
structure of a $\Gamma$-module, where $\Gamma=S_n^{(2)}$ is the group of
permutations of the edges of the complete graph $K_n$ induced by the
renumbering of its vertices. This action leads to orbit sums and to the
algebra of graph invariants $\mathcal A^\Gamma$. On the other hand, on the
Boolean lattice of edge subsets there naturally act the operators of adding
and deleting one edge, which define an action of the complex Lie algebra
$\mathfrak{sl}_2$ on $\mathcal A$.

This $\mathfrak{sl}_2$-structure is natural in the broader context of graded
partially ordered sets, Lefschetz operators, and the Sperner property. Related
ideas were developed, in particular, in the works of Stanley and Proctor
\cite{Stanley1980,Proctor1982}. Especially close to our approach is
Mnukhin's $\mathfrak{sl}_2$-description of Boolean and related graded posets
\cite{Mnukhin1993}, where raising and lowering operators are used to obtain
combinatorial consequences, in particular the unimodality of rank numbers and
the Sperner property. In this paper we consider the graph-theoretic case of
this mechanism, where the elements of the Boolean poset are subsets of the
edge set of $K_n$, but we shift the emphasis from rank properties to the
representation-theoretic decomposition of the graph algebra, primitive
components, and their compatibility with the action of the pair group
$\Gamma$.

Thus, in this paper the graph algebra is considered as a finite-dimensional
$\mathfrak{sl}_2$-module compatible with the action of the group $\Gamma$.
The task of the paper is to give a representation-theoretic description of
the graph algebra taking both of these structures into account.

We first introduce two natural operators $D_-$ and $D_+$, which correspond
respectively to deleting and adding one edge and thereby define on
$\mathcal A$ the structure of an $\mathfrak{sl}_2$-module. We then describe
the decomposition of $\mathcal A$ into irreducible $\mathfrak{sl}_2$-modules.
A central role in this description is played by the primitive spaces
$$
P_k=\ker(D_-:A_k\to A_{k-1}),
$$
which give the initial components of the corresponding irreducible chains.
As a result, the rank structure of the graph algebra receives a
representation-theoretic interpretation: the subspaces $A_k$, corresponding
to graphs with a fixed number of edges, are described through primitive
components and their $\mathfrak{sl}_2$-generations. Separately, we describe
the module of $\mathfrak{sl}_2$-invariants
$\mathcal A^{\mathfrak{sl}_2}$ and prove that it is zero for odd $m$, while
for even $m$ it coincides with $P_{m/2}$.

As the second step, we study the action of the symmetric group on the edge
coordinates. For $0\le k\le \lfloor m/2\rfloor$, the space $A_k$ is
identified with the permutation $S_m$-module on $k$-element subsets, and the
primitive space $P_k$ has a natural interpretation as the two-row Specht module
of shape $(m-k,k)$. We prove that the algebra of $\Gamma$-invariants has a
monogenic structure and obtain its decomposition
$$
A_r^\Gamma
=
\bigoplus_{k=0}^{\min(r,m-r)}
D_+^{\,r-k}P_k^\Gamma,
$$
which describes each rank subspace of graph invariants through primitive
$\Gamma$-invariant components.

The combination of these two descriptions is carried out through Schur--Weyl
duality. The first step is to establish the tensor interpretation of the graph
algebra as an $\mathfrak{sl}_2$-module:
$$
\mathcal A\cong V_1^{\otimes m},
$$
where $V_1$ is the standard two-dimensional $\mathfrak{sl}_2$-module. It is
then shown that the graph algebra has the decomposition
$$
\mathcal A
\cong
\bigoplus_{k=0}^{\lfloor m/2\rfloor}
S^{(m-k,k)}\otimes V_{m-2k},
$$
where $S^{(m-k,k)}$ is a Specht module and $V_{m-2k}$ is an irreducible
$\mathfrak{sl}_2$-module. This decomposition shows that the
$\mathfrak{sl}_2$-structure and the permutation structure on the edges are not
independent: they are two sides of one tensor description of the algebra
$\mathcal A$. In particular, it gives a representa\-tion-theoretic explanation
of the orbital enumeration of graphs: the dimensions of the invariant
components $A_k^\Gamma$ coincide with the number of non-isomorphic simple
graphs with $n$ vertices and $k$ edges, which leads to a representation-theoretic
interpretation of the Burnside--Polya formula.

The paper is organized as follows. In Section~2 we introduce the main objects:
the pair group, the graph algebra, its rank structure, and the necessary facts
about $\mathfrak{sl}_2$-modules. In Section~3 we construct the
$\mathfrak{sl}_2$-structure on $\mathcal A$ and describe the primitive spaces.
In Section~4 we consider the action of the group $\Gamma$, orbit sums, and the
algebra of invariants. In Section~5 we establish the connection with
Schur--Weyl duality and derive representation-theoretic formulas for counting
graph orbits.

\section{Preliminaries}


In this section we introduce the main combinatorial and algebraic objects
which underlie the subsequent study.

\subsection{The induced edge action of $S_n$ and the pair group}\label{ss2.1}

Let $V=\{1,2,\ldots,n\}$, and let
$$
E_n=V^{(2)}=\{\{i,j\}:1\le i<j\le n\},
$$
be the edge set of the complete graph $K_n$.
Every permutation $\pi\in S_n$ induces a permutation of the set $E_n$ by the
rule
$$
\{i,j\}\mapsto \{\pi(i),\pi(j)\}.
$$
This gives a homomorphism
$$
\varphi:S_n\longrightarrow S_{E_n}\cong S_{m}, \qquad m=\binom{n}{2},
$$
whose image is called the \textit{pair group} $S_n^{(2)}$:
$$
S_n^{(2)}=\varphi(S_n)\subset S_{m},
$$
see also \cite[Chapter 4]{H-P}.

\begin{theorem}\label{phi_iso}
  If $n \ge 3$, then $\varphi : S_n \longrightarrow  S_n^{(2)}$ is an
  isomorphism.
\end{theorem}

\begin{proof}
  Surjectivity follows from the definition of the image.
  Let us prove that the kernel is trivial.
  Let $\sigma \in \ker\varphi$, that is,
  $\{\sigma(i),\sigma(j)\} = \{i,j\}$ for all $i \ne j$.
  Fix an arbitrary $i \in V$.
  Since $n \ge 3$, there exist $j, l \in V\setminus\{i\}$ with $j \ne l$,
  and therefore
  $$
    \sigma(i) \;\in\; \{i,j\} \cap \{i,l\} \;=\; \{i\}.
  $$
  Hence $\sigma(i)=i$.
  Since $i$ is arbitrary, we have $\sigma=\mathrm{id}$ and
  $\ker\varphi=\{1\}$.
\end{proof}

Thus, for $n\ge 3$, the group $S_n^{(2)}$ is a subgroup of order $n!$ of the
group of permutations of the edge set of the complete graph $K_n$, induced by
the action of the symmetric group $S_n$ on its vertices; see also
\cite[Chapter~4]{H-P}.

\subsection{The graph algebra $\mathcal A$ and its rank structure}


For each pair $1\le i<j\le n$, define the edge coordinate function $x_{ij}$
on the set of all simple graphs on the vertex set $V$ by the rule
$$
x_{ij}(G)=
\begin{cases}
1,& \{i,j\}\in E(G),\\
0,& \{i,j\}\notin E(G),
\end{cases}
$$
where $E(G)\subseteq V^{(2)}$ is the edge set of the graph $G$.
The coordinate functions form the algebra of polynomial functions on the set
of all graphs on $n$ vertices.

Define the action of the group $S_n^{(2)}$ on the edge coordinates in the
standard way:
$$
\sigma \cdot x_{i\,j} = x_{\sigma^{-1}(i)\,\sigma^{-1}(j)}, \quad \sigma \in S_n^{(2)}.
$$

In order to avoid double-index notation for coordinate functions, fix the
lexicographic order on the set of all pairs of vertices $V^{(2)}$ and number
the corresponding edge coordinate functions by a single index:
$$
\{x_1,x_2,\ldots,x_m\}=\{x_{12},x_{13},\ldots,x_{n-1,n}\}.
$$
After this relabelling, each edge of a graph is identified with some index
$i\in [m]=\{1,2,\ldots,m\}$, and the edge set $E(G)$ of an arbitrary graph
on $n$ vertices will be regarded as a subset of $[m]$.
Every renumbering of the vertices induces a renumbering of the edges, that is,
a permutation of the set $[m]$; hence the pair group $S_n^{(2)}$ acts on the
set of edge indices.
We shall regard this induced permutation group as a subgroup of the symmetric
group $S_m$ and will denote it below by $\Gamma$; of course,
$\Gamma\cong S_n^{(2)}$.

From the functional point of view, for simple unweighted graphs each variable
$x_i$ is Boolean, and therefore the identity $x_i^2=x_i$ holds on the set
$\{0,1\}^m$ for all $i$.
Thus it is natural to consider the quotient algebra
$$
\mathcal A=\mathbb{C}[x_1,\dots,x_m]/(x_1^2-x_1,\dots,x_m^2-x_m),
$$
which is called the \emph{graph algebra} on $n$ vertices and will be the main
object of our study.

As a vector space, $\mathcal A$ has dimension $2^m$, and its natural basis is
formed by the classes of squarefree monomials
$$
\boldsymbol{x}_S:=\prod_{e\in S}x_e,\qquad S\subseteq [m].
$$
Multiplication in $\mathcal A$ is given by the rule
$
\boldsymbol{x}_S \boldsymbol{x}_T=\boldsymbol{x}_{S\cup T}.
$
The algebra $\mathcal A$ is not graded, but it has a natural rank structure by
degree:
$$
\mathcal A=\bigoplus_{k=0}^m A_k,
$$
where
$
A_k:=\operatorname{span}\{\boldsymbol{x}_S:\ |S|=k\}, \dim A_k=\binom{m}{k}.
$

\medskip
Let us show that the basis monomials in $\mathcal A$ are naturally
interpreted as simple graphs.
The value of the monomial $\boldsymbol{x}_S$ on a graph $G$ is computed as the
product of the values of the corresponding edge variables:
$$
\boldsymbol{x}_S(G):=\prod_{e\in S} x_e(G).
$$
Since $x_e(G)\in\{0,1\}$, the condition $\boldsymbol{x}_S(G)=1$ is equivalent
to the fact that \emph{all} edges from $S$ are present in $G$, that is,
$S\subseteq E(G)$, while if at least one edge from $S$ is absent from $G$, then
$\boldsymbol{x}_S(G)=0$.
In particular, for every edge set $S$ there exists a unique graph $G_S$ with
edge set $E(G)=S$ on which $\boldsymbol{x}_S(G_S)=1$.
Therefore, below we shall identify graphs with the corresponding basis
monomials $\boldsymbol{x}_S$ of the algebra $\mathcal A$.
Under this identification, the complementary graph $\overline{G}$ corresponds
to the monomial $\boldsymbol{x}_{\overline{S}}$, where
$\overline{S}=[m]\setminus S$.

\begin{example}\rm
For $n=4$ we have $m=6$ and fix the following relabelling of the edge
variables:
$$
\{x_1,x_2,x_3,x_4,x_5,x_6\}
=\{x_{12},x_{13},x_{14},x_{23},x_{24},x_{34}\},
$$
that is,
\begin{gather}\label{perepoz}
x_{12}\mapsto x_1,\quad x_{13}\mapsto x_2,\quad x_{14}\mapsto x_3,\quad
x_{23}\mapsto x_4,\quad x_{24}\mapsto x_5,\quad x_{34}\mapsto x_6.
\end{gather}
The basis of the algebra $\mathcal A$ as a vector space consists of all
squarefree monomials
$$
\boldsymbol{x}_S=\prod_{i\in S}x_i,\qquad S\subseteq [6]=\{1,2,3,4,5,6\},
$$
that is, of $2^6$ elements in total:
$$
1,\ x_1,\ldots,x_6,\ x_1x_2,\ldots,\ x_1x_2x_3x_4x_5x_6.
$$
Multiplication in $\mathcal A$ is governed by the idempotent relations
$x_i^2=x_i$; for example,
$$
(x_1+x_2+x_3)^2
=x_1^2+x_2^2+x_3^2+2\,(x_1 x_2+x_1 x_3+x_2 x_3)
=
x_1+x_2+x_3+2(x_1 x_2+x_1 x_3+x_2 x_3).
$$
The cycle $C_4$ on the vertices $\{1,2,3,4\}$ with edges
$\{\{1,2\},\{2,3\},\{3,4\},\{1,4\}\}$, which in our numbering correspond
to the indices $\{1,4,6,3\}$, is identified with the monomial
$x_1x_3x_4x_6$.
$\triangle$ \end{example}

\medskip
As we see, the graph algebra admits two complementary viewpoints. On the one
hand, it is an algebra of squarefree polynomials; on the other hand,
$\mathcal A$ is canonically identified with the algebra of all pseudo-Boolean
functions on the Boolean cube.
In what follows, we shall consider the graph algebra from these two points of
view whenever needed.

The ring structure of the algebra $\mathcal A$ is very simple: since the
characteristic of the field $\mathbb C$ is zero, the algebra $\mathcal A$ is a
finite-dimensional commutative semisimple (etale) algebra and is isomorphic to
the direct product $\mathbb C^{2^m}$; hence, by the primitive element theorem,
it is generated by one element~\cite[Theorem~2.2]{Dobbs1999}, see
also~\cite[Chapters~II, V]{LangAlgebra}.

In the next theorem this generating element is found explicitly.

\begin{theorem}\label{A_monogenic}
The algebra
$
\mathcal A
$
is monogenic as an abstract $\mathbb C$-algebra and is generated by the element
$$
z=x_1+2x_2+4x_3+\cdots+2^{m-1}x_m,
$$
that is,
$
\mathcal A=\mathbb C[z].
$
\end{theorem}

\begin{proof}
Consider the element $z$ as a function from
$
\operatorname{Fun}(\{0,1\}^m,\mathbb C).
$
To each point
$$
\varepsilon=(\varepsilon_1,\dots,\varepsilon_m)\in\{0,1\}^m,
$$
there corresponds the value
$$
z(\varepsilon)
=
\varepsilon_1+2\varepsilon_2+4\varepsilon_3+\cdots+2^{m-1}\varepsilon_m.
$$
This is the binary code of the point $\varepsilon$, and hence all values
$z(\varepsilon)$ are pairwise distinct.

For each basis monomial $x_S$, put
$$
N(S):=\sum_{i\in S}2^{i-1}.
$$
Then $z$ takes the value $N(S)$ at the point corresponding to the subset $S$.
For each $S$, construct the Lagrange polynomial
$$
L_S(t):=
\prod_{\substack{T\subseteq[m]\\T\neq S}}
\frac{t-N(T)}{N(S)-N(T)}.
$$
After substituting $t=z$, we obtain an element
$
L_S(z)\in \mathbb C[z],
$
which is equal to $1$ at the point $S$ and to $0$ at all other points of the
Boolean cube.

Thus, all characteristic functions of points of the Boolean cube belong to
$\mathbb C[z]$.
Since these characteristic functions form a basis of the algebra
$
\operatorname{Fun}(\{0,1\}^m,\mathbb C)\cong \mathcal A,
$
we have
$
\mathcal A=\mathbb C[z].
$
\end{proof}

\begin{example}{\rm
For $n=3$, direct computations give
\begin{gather*}
x_1={\frac {1}{315}}\,z \left( z-6 \right)  \left( z-2 \right)  \left( z-4
 \right)  \left( 4\,{z}^{3}-50\,{z}^{2}+176\,z-151 \right) ,\\
x_2= {\frac {1}{1260}}\,z \left( z-1 \right)  \left( z-4 \right)  \left( z-
5 \right)  \left( 2\,{z}^{3}-29\,{z}^{2}+107\,z-9 \right) ,\\
x_3=-{\frac {1}{2520}}\,z \left( z-1 \right)  \left( z-2 \right)  \left( z
-3 \right)  \left( 10\,{z}^{3}-185\,{z}^{2}+1139\,z-2341 \right),
\end{gather*}
where $z=x_1+2 x_2 +4 x_3$.
}
$\triangle$ \end{example}

In the general case, the expressions for the basic variables are very
cumbersome polynomials in $z$ of degree $2^{m-1}$, which require additional
reduction modulo the ideal
$I=(x_1^2-x_1, \ldots, x_m^2-x_m).$


\subsection{The Lie algebra $\mathfrak{sl}_2$
  and its standard irreducible modules}

In this subsection we briefly recall standard facts about the Lie algebra
$\mathfrak{sl}_2$ and its finite-dimensional modules, on which the subsequent
analysis of the graph algebra~$\mathcal{A}$ is based.
All the statements given below are classical; their proofs can be found in
\cite[Chapter~11]{FH} or in \cite[Chapters~7--8]{Humphreys}.
All vector spaces are considered over the field $\mathbb{C}$.

\medskip

The Lie algebra $\mathfrak{sl}_2 = \mathfrak{sl}_2(\mathbb{C})$ is the space
of all $2 \times 2$ matrices of trace zero over~$\mathbb{C}$:
$$
  \mathfrak{sl}_2
  = \bigl\{X \in M_2(\mathbb{C}) : \operatorname{tr}(X) = 0\bigr\},
$$
where the Lie bracket is given by the matrix commutator
$[X, Y] = XY - YX$.
This algebra is three-dimensional and has the standard basis
$$
  e_+ = \begin{pmatrix} 0 & 1 \\ 0 & 0 \end{pmatrix}, \qquad
  e_- = \begin{pmatrix} 0 & 0 \\ 1 & 0 \end{pmatrix}, \qquad
  h   = \begin{pmatrix} 1 & 0 \\ 0 & -1 \end{pmatrix},
$$
which satisfies the commutation relations
\begin{equation}\label{sl2_relations}
  [h, e_+] = 2e_+, \qquad
  [h, e_-] = -2e_-, \qquad
  [e_+, e_-] = h.
\end{equation}

\medskip

A \emph{representation} of the Lie algebra $\mathfrak{sl}_2$ in a vector
space~$V$ is a homomorphism of Lie algebras
$\rho : \mathfrak{sl}_2 \to \operatorname{End}(V)$;
in this case $V$ is called an \emph{$\mathfrak{sl}_2$-module}.
The elements $e_+$ and $e_-$ are called, respectively, the \emph{raising
operator} and the \emph{lowering operator}; in what follows, instead of
$\rho(e_\pm)$ and $\rho(h)$ we write the same symbols.

A nonzero vector $v \in V$ is called a \emph{weight vector of weight}
$\lambda \in \mathbb{C}$ if $hv = \lambda v$; the corresponding subspace
$V_\lambda = \ker(h - \lambda \cdot \mathrm{id}) \subseteq V$
is called the \emph{weight subspace of weight~$\lambda$}.
It follows directly from the relations~\eqref{sl2_relations} that
$$
  e_\pm V_\lambda \subseteq V_{\lambda \pm 2},
$$
that is, $e_+$ raises the weight by~$2$, while $e_-$ lowers it by~$2$.
A nonzero vector $v_0 \in V_\lambda$ satisfying $e_- v_0 = 0$ is called a
\emph{lowest weight vector}; the number $\lambda$ is called the
\emph{lowest weight} of the given module.

\medskip

An $\mathfrak{sl}_2$-module $V$ is called \emph{irreducible} if it contains
no nontrivial proper $\mathfrak{sl}_2$-submodules.
For each integer $d \ge 0$ there exists a standard irreducible
$\mathfrak{sl}_2$-module~$V_d$ of dimension $d+1$ with basis
$v_0, v_1, \dots, v_d$, on which the action is given by the formulas:
\begin{gather*}
  e_+ v_i = (d-i)\, v_{i+1}, \quad
  h\, v_i = (2i-d)\, v_i, \quad 
  e_- v_i = i\, v_{i-1},
\end{gather*}
where $v_{-1} = v_{d+1} = 0$.
The vector $v_0$ is a lowest weight vector of weight~$-d$, and $v_d$ is a
highest weight vector of weight~$d$.
The ordered weights of the module~$V_d$ form the arithmetic progression
$$
  -d,\; -d+2,\; \ldots,\; d-2,\; d,
$$
and each weight subspace is one-dimensional.
In particular, the lowest weight~$-d$ uniquely determines the irreducible
module of dimension $d+1$.

The next two theorems record standard structural facts about
finite-dimensional $\mathfrak{sl}_2$-modules over an algebraically closed
field of characteristic zero.

\begin{theorem}
  Every finite-dimensional irreducible $\mathfrak{sl}_2$-module over a field
  of characteristic zero is isomorphic to exactly one of the standard
  modules~$V_d$, $d \in \mathbb{Z}_{\ge 0}$.
\end{theorem}

\begin{theorem}\label{complete_reducibility_sl2}
  Every finite-dimensional $\mathfrak{sl}_2$-module over a field of
  characteristic zero is isomorphic to a direct sum of irreducible
  submodules:
  $$
    V \cong \bigoplus_{d \ge 0} m_d\, V_d,
  $$
  where the multiplicities $m_d \in \mathbb{Z}_{\ge 0}$, almost all of which
  are zero, are uniquely determined.
\end{theorem}

\medskip
The decomposition of a module~$V$ into irreducible submodules is practically
found using elements of the kernel $\ker e_-$ of the operator $e_-$.
Each nonzero vector $v_0 \in \ker e_- \cap V_{d}$ generates an irreducible
submodule
$\langle v_0, e_+v_0, e_+^2v_0, \dots \rangle \cong V_d$;
the multiplicity $m_d$ is equal to the dimension of the corresponding weight
subspace:
$$
  m_d = \dim(\ker e_- \cap V_{d}).
$$
Thus, the complete information about the decomposition is encoded in the
structure of $\ker e_-$.

\medskip

In the next section we define on the graph algebra~$\mathcal{A}$ the
structure of an $\mathfrak{sl}_2$-module.

\section{The graph algebra $\mathcal A$ as an $\mathfrak{sl}_2$-module}

In this section we show that the graph algebra $\mathcal A$ is naturally
endowed with the structure of a finite-dimensional $\mathfrak{sl}_2$-module.
The idea of this construction is quite simple: if a squarefree monomial
$\boldsymbol{x}_S$ is interpreted as a graph with edge set $S\subseteq[m]$,
then the natural operations on such monomials are adding one new edge,
deleting one existing edge, and recording the number of edges.
These three operations lead to operators which, as it turns out, satisfy the
standard commutation relations of the algebra $\mathfrak{sl}_2$.
As a result, the rank structure of the algebra $\mathcal A$ is compatible
with its weight decomposition, and the problem of describing $\mathcal A$
is reduced to classical questions in the representation theory of
$\mathfrak{sl}_2$.

\subsection{The action of the algebra $\mathfrak{sl}_2$ on $\mathcal {A}$}

Consider on $\mathcal {A}$ three linear operators defined on basis monomials as follows:
\begin{gather*}
D_+(\boldsymbol{x}_S)=\sum_{e\in\overline S}\boldsymbol{x}_{S\cup\{e\}}, D_+(\boldsymbol{x}_{[m]})=0,\\
D_0(\boldsymbol{x}_S)=-(m-2|S|)\,\boldsymbol{x}_S,\\
D_-(\boldsymbol{x}_S)=\sum_{e\in S}\boldsymbol{x}_{S\setminus\{e\}},D_-(\boldsymbol{x}_\varnothing)=0.
\end{gather*}

All three operators are compatible with the rank structure on $\mathcal A:$
$$
D_+:A_k\to A_{k+1},\qquad
D_-:A_k\to A_{k-1},\qquad
D_0:A_k\to A_k.
$$

We now show that these three operators $D_+,D_0,D_-$ define the standard
action of the algebra $\mathfrak{sl}_2$ on $\mathcal A$.
\begin{theorem}
The operators $D_+,D_0,D_-$ satisfy the relations
\begin{equation}\label{sl2_on_A_relations}
[D_+,D_-]=D_0,\qquad
[D_0,D_+]=2D_+,\qquad
[D_0,D_-]=-2D_-,
\end{equation}
and define on the algebra $\mathcal A$ the structure of an
$\mathfrak{sl}_2$-module.
\end{theorem}

\begin{proof}
It is enough to check the relations \eqref{sl2_on_A_relations} on basis
monomials $\boldsymbol{x}_S, |S|=k$, and then use linearity.

\smallskip

First we verify the relation
$
[D_0,D_+]=2D_+.
$
Every monomial in the sum
$$
D_+(\boldsymbol{x}_S)=\sum_{e\in\overline S}\boldsymbol{x}_{S\cup\{e\}},
$$
has degree $k+1$, and therefore
$$
D_0(D_+(\boldsymbol{x}_S))=-(m-2(k+1))D_+(\boldsymbol{x}_S).
$$
On the other hand,
$$
D_+(D_0(\boldsymbol{x}_S))=D_+\bigl(-(m-2k)\boldsymbol{x}_S\bigr)=-(m-2k)D_+(\boldsymbol{x}_S).
$$
Hence
$$
[D_0,D_+](\boldsymbol{x}_S)
=
D_0(D_+(\boldsymbol{x}_S))-D_+(D_0(\boldsymbol{x}_S))
=
\bigl(-(m-2k-2)+(m-2k)\bigr)D_+(\boldsymbol{x}_S)
=
2D_+(\boldsymbol{x}_S).
$$

The relation $[D_0,D_-]=-2D_-$ is verified similarly:
$$
[D_0,D_-](\boldsymbol{x}_S)
=
D_0(D_-(\boldsymbol{x}_S))-D_-(D_0(\boldsymbol{x}_S))
=
\bigl(-(m-2k+2)+(m-2k)\bigr)D_-(\boldsymbol{x}_S)
=
-2D_-(\boldsymbol{x}_S).
$$

It remains to verify the last relation
$
[D_+,D_-]=D_0.
$
We have
$$
D_-D_+(\boldsymbol{x}_S)
=
\sum_{f\in\overline S}D_-(\boldsymbol{x}_{S\cup\{f\}}).
$$
For a fixed $f\in\overline S$ we obtain
$$
D_-(\boldsymbol{x}_{S\cup\{f\}})
=
\sum_{e\in S\cup\{f\}}\boldsymbol{x}_{(S\cup\{f\})\setminus\{e\}}
=
\boldsymbol{x}_S+\sum_{e\in S}\boldsymbol{x}_{(S\setminus\{e\})\cup\{f\}}.
$$
Therefore,
$$
D_-D_+(\boldsymbol{x}_S)
=
(m-k)\boldsymbol{x}_S+\sum_{f\in\overline S}\sum_{e\in S}\boldsymbol{x}_{(S\setminus\{e\})\cup\{f\}}.
$$

On the other hand,
$$
D_-(\boldsymbol{x}_S)=\sum_{e\in S}\boldsymbol{x}_{S\setminus\{e\}},
$$
and hence
$$
D_+D_-(\boldsymbol{x}_S)
=
\sum_{e\in S}D_+(\boldsymbol{x}_{S\setminus\{e\}}).
$$
For a fixed $e\in S$, the set of edges absent from $S\setminus\{e\}$ is
$$
[m]\setminus (S\setminus\{e\})=\overline S\cup\{e\},
$$
and therefore
\begin{equation}\label{r2}
D_+(\boldsymbol{x}_{S\setminus\{e\}})
=
\boldsymbol{x}_S+\sum_{f\in\overline S}\boldsymbol{x}_{(S\setminus\{e\})\cup\{f\}}.
\end{equation}
It follows that
$$
D_+D_-(\boldsymbol{x}_S)
=
k\boldsymbol{x}_S+\sum_{e\in S}\sum_{f\in\overline S}\boldsymbol{x}_{(S\setminus\{e\})\cup\{f\}}.
$$

The double sums in the last two formulas coincide, since they are taken over
the same pairs
$$
(e,f)\in S\times\overline S.
$$
Thus, after subtraction, they cancel each other, and we obtain
$$
[D_+,D_-](\boldsymbol{x}_S)
=
D_+D_-(\boldsymbol{x}_S)-D_-D_+(\boldsymbol{x}_S)
=
(2k-m)\boldsymbol{x}_S
=
-(m-2 |S|)\boldsymbol{x}_S
=
D_0(\boldsymbol{x}_S).
$$

Hence all relations \eqref{sl2_on_A_relations} hold on the basis, and
therefore on the whole algebra $\mathcal A$.
\end{proof}
The combinatorial meaning of the operators $D_+,D_-,D_0$ is completely
transparent: $D_+$ runs through all one-edge supergraphs of a given graph,
$D_-$ runs through all its one-edge subgraphs, and $D_0$ records the number
of edges.
It is precisely this interpretation that makes the constructed
$\mathfrak{sl}_2$-structure natural from the point of view of graph theory.

\subsection{Primitive elements}

The elements of the algebra $\mathcal A$ which are annihilated by the operator
$D_-$ form the subspace
$$
\ker D_-=\{f\in\mathcal A:\ D_-f=0\}.
$$
The kernel of the operator $D_-$ is a rank subspace in $\mathcal A$
$$
\ker D_-= P_0+P_1+ \cdots + P_m,
$$
where
$$
P_k=\ker\left(D_-:A_k\longrightarrow A_{k-1}\right),
\qquad A_{-1}=0.
$$
Elements of the space $P_k$ will be called \emph{primitive elements of degree}
$k$.

The following theorem shows that primitive elements are initial vectors of
irreducible $\mathfrak{sl}_2$-chains and that such elements can appear only
in the lower half of the algebra.
\begin{theorem}\label{primitive_chain}
Let $0\neq z\in P_k$. If $k\le \lfloor m/2\rfloor$, then the elements
$$
z,\ D_+(z),\ D_+^2(z),\dots,\ D_+^{m-2k}(z),
$$
generate an irreducible $\mathfrak{sl}_2$-module isomorphic to $V_{m-2k}$.
In particular,
$$
D_+^{m-2k}(z)\neq 0,
\qquad
D_+^{m-2k+1}(z)=0.
$$
If $k>\lfloor m/2\rfloor$, then $P_k=0$.
\end{theorem}

\begin{proof}
Since $\mathcal A$ is a finite-dimensional $\mathfrak{sl}_2$-module, by the
complete reducibility theorem it is a direct sum of irreducible
$\mathfrak{sl}_2$-modules. Let $0\neq z\in P_k$. Then $z\in A_k$,
$D_-z=0$, and the weight of $z$ is
$$
-(m-2k).
$$
Thus, $z$ is a lowest weight vector in the submodule that it generates.
Therefore this submodule is isomorphic to the standard irreducible module
$V_{m-2k}$. It follows that the chain
$$
z,\ D_+z,\ D_+^2z,\dots,\ D_+^{m-2k}z,
$$
has length $m-2k+1$, and hence
$$
D_+^{m-2k}(z)\neq 0,
\qquad
D_+^{m-2k+1}(z)=0.
$$

If $k>m/2$, then the weight $2k-m$ is positive. But a nonzero vector
annihilated by the operator $D_-$ in a finite-dimensional
$\mathfrak{sl}_2$-module can only be a lowest weight vector of some
irreducible component, and the lowest weights of irreducible modules are of
the form $-d\le 0$. Hence, for $k>m/2$, there are no such vectors, that is,
$P_k=0$.
\end{proof}

The following theorem gives the dimension of the vector space of primitive
elements $P_k$.

\begin{theorem}\label{primitive_dimensions}
For each $k=0,1,\dots,m$ we have
$$
\dim P_k=
\begin{cases}
\displaystyle \binom{m}{k}-\binom{m}{k-1},
& 0\le k\le \left\lfloor \dfrac m2\right\rfloor,\\[8pt]
0,
& \left\lfloor \dfrac m2\right\rfloor<k\le m.
\end{cases}
$$
\end{theorem}

\begin{proof}
Since $\mathcal A$ is a finite-dimensional $\mathfrak{sl}_2$-module over a
field of characteristic zero, it decomposes into a direct sum of irreducible
$\mathfrak{sl}_2$-modules.

In every irreducible module generated by a primitive element of degree $j$,
there is exactly one vector in each degree
$$
j,\ j+1,\dots,\ m-j,
$$
and there are no vectors in other degrees. Moreover, each such chain has a
unique initial vector annihilated by the operator $D_-$, while the chains that
started in smaller degrees have already been counted in the preceding
subspaces. Hence, if $k\le \lfloor m/2\rfloor$, the difference
$$
\dim A_k-\dim A_{k-1},
$$
counts exactly the number of new irreducible chains that start in degree
$k$. This number is equal to $\dim P_k$, because each such chain has exactly
one primitive vector in its initial degree.
Therefore
$$
\dim P_k
=
\dim A_k-\dim A_{k-1}
=
\binom{m}{k}-\binom{m}{k-1}.
$$
For $k>\lfloor m/2\rfloor$, the equality $P_k=0$ has already been established
in Theorem~\ref{primitive_chain}.
\end{proof}

As a consequence, for $1\le k\le \lfloor m/2\rfloor$ the operator
$
D_-
$
is surjective. Indeed,
$$
\operatorname{rank} D_-|_{A_k}
=
\dim A_k-\dim P_k
=
\binom{m}{k}
-
\left(\binom{m}{k}-\binom{m}{k-1}\right)
=
\binom{m}{k-1}
=
\dim A_{k-1}.
$$

Thus, for $k\le \lfloor m/2\rfloor$, the elements of the space $P_k$ are
precisely those elements in the space $A_k$ which are not obtained by the
action of the operator $D_+$ on elements of $A_{k-1}$.
It follows from the irreducible $\mathfrak{sl}_2$-decomposition that the
operator $D_+$ is injective on subspaces below the middle level. Therefore,
for $k\le \lfloor m/2\rfloor$, we have the direct sum
$$
A_k=P_k\oplus D_+(A_{k-1}).
$$

Iterating this decomposition, we obtain a decomposition of $A_k$ through
primitive spaces of smaller degrees:
\begin{equation} \label{P_to_A}
A_k
=
\bigoplus_{j=0}^{\min(k,m-k)}
D_+^{\,k-j}P_j.
\end{equation}

In other words, every element of $A_k$ is a sum of contributions from those
$\mathfrak{sl}_2$-chains which start in some primitive space $P_j$ and reach
$A_k$. The condition
$j\le \min(k,m-k)$ precisely means that the chain generated by $P_j$ contains
a component of degree $k$.

\medskip

Note that the operators $D_-$ and $D_+$ are not derivations of the graph
algebra. For the operator $D_-$ this is seen from the following formula:
$$
D_-(\boldsymbol{x}_S\boldsymbol{x}_T)
=
D_-(\boldsymbol{x}_S)\boldsymbol{x}_T
+
\boldsymbol{x}_S D_-(\boldsymbol{x}_T)
-
2|S\cap T|\,\boldsymbol{x}_{S\cup T}
+
\sum_{e\in S\cap T}
\boldsymbol{x}_{(S\cup T)\setminus\{e\}}.
$$
However, if $S\cap T=\varnothing$, then
\begin{equation}\label{Leibniz}
D_-(\boldsymbol{x}_S\boldsymbol{x}_T)=D_-(\boldsymbol{x}_S)\boldsymbol{x}_T+\boldsymbol{x}_SD_-(\boldsymbol{x}_T).
\end{equation}
Thus, the operator $D_-$ behaves as a derivation only on products of monomials
with disjoint supports.

\medskip

Let us describe all primitive elements for $n=4$.

\begin{example}
{\rm
By Theorem~\ref{primitive_dimensions} we have
$$
\dim P_0=1,\qquad
\dim P_1=5,\qquad
\dim P_2=9,\qquad
\dim P_3=5,
$$
and
$$
P_k=0,\qquad k>3.
$$
Let us describe these vector spaces of primitive elements explicitly.

\medskip

\noindent\textbf{Degree $0$.}
Since $A_0=\langle 1\rangle$ and $D_-(1)=0$, we have
$
P_0=\langle 1\rangle.
$
The element $1$ has weight $-6$ and generates an irreducible
$\mathfrak{sl}_2$-submodule isomorphic to $V_6$.

\medskip

\noindent\textbf{Degree $1$.}
On basis monomials of degree $1$ we have
$
D_-(x_i)=1,\qquad i=1,\dots,6,
$
and therefore
$$
P_1=\left\{\sum_{i=1}^6 a_i x_i:\ \sum_{i=1}^6 a_i=0\right\}.
$$
A convenient basis of this space is
$$
P_1=
\left\langle
x_1-x_2,\ x_1-x_3,\ x_1-x_4,\ x_1-x_5,\ x_1-x_6
\right\rangle.
$$
Every nonzero element of $P_1$ generates an irreducible module of type $V_4$.

\medskip

\noindent\textbf{Degree $2$.}
The space $P_2=\ker(D_-:A_2\to A_1)$ has dimension
$$
\dim P_2=\binom62-\binom61=9.
$$
Its basis can be chosen in the form of products of linear differences:
$$
\begin{aligned}
f_1&=(x_1-x_2)(x_3-x_4),&
f_2&=(x_1-x_2)(x_3-x_5),&
f_3&=(x_1-x_2)(x_3-x_6),\\
f_4&=(x_1-x_3)(x_2-x_4),&
f_5&=(x_1-x_3)(x_2-x_5),&
f_6&=(x_1-x_3)(x_2-x_6),\\
f_7&=(x_1-x_4)(x_2-x_5),&
f_8&=(x_1-x_4)(x_2-x_6),&
f_9&=(x_1-x_5)(x_2-x_6).
\end{aligned}
$$
Their primitiveness follows immediately from \eqref{Leibniz}, and their
linear independence follows from the fact that each $f_i$ contains a monomial
which does not occur in the expansion of any other element in the list.
Every nonzero element of $P_2$ generates an irreducible
$\mathfrak{sl}_2$-module of type $V_2$.

\medskip

\noindent\textbf{Degree $3$.}
Primitive elements of degree $3$ have weight $0$ and generate trivial
$\mathfrak{sl}_2$-modules of type $V_0$. The dimension of this space is
$$
\dim P_3=\binom63-\binom62=5.
$$
A convenient basis is given by five products of three linear differences:
$$
\begin{aligned}
g_1&=(x_1-x_2)(x_3-x_4)(x_5-x_6),\\
g_2&=(x_1-x_2)(x_3-x_5)(x_4-x_6),\\
g_3&=(x_1-x_3)(x_2-x_4)(x_5-x_6),\\
g_4&=(x_1-x_3)(x_2-x_5)(x_4-x_6),\\
g_5&=(x_1-x_4)(x_2-x_5)(x_3-x_6).
\end{aligned}
$$
Since this is a trivial $\mathfrak{sl}_2$-module, for every $g\in P_3$ we
also have
$
D_+(g)=0.
$

Thus, for $n=4$ the primitive spaces have the form
$$
P_0\cong \mathbb C,\qquad
\dim P_1=5,\qquad
\dim P_2=9,\qquad
\dim P_3=5,
$$
and the corresponding decomposition of the algebra as an
$\mathfrak{sl}_2$-module is
$$
\mathcal A
\cong
V_6\oplus 5V_4\oplus 9V_2\oplus 5V_0.
$$
Checking the dimension gives
$$
\dim\mathcal A
=
1\cdot 7+5\cdot 5+9\cdot 3+5\cdot 1
=
2^6.
$$
}
$\triangle$ \end{example}


In the general case, the following statement holds:

\begin{theorem}\label{irred_decomp_A}
The algebra $\mathcal A$, as an $\mathfrak{sl}_2$-module, decomposes into a
direct sum of irreducible modules
\begin{equation}
\mathcal A
\cong
\bigoplus_{k=0}^{\lfloor m/2\rfloor}
\mu_k\,V_{m-2k},
\end{equation}
where
\begin{equation}
\mu_k
=
\dim P_k
=
\binom{m}{k}-\binom{m}{k-1},
\qquad
0\le k\le \left\lfloor\frac m2\right\rfloor, \quad \binom{m}{-1}=0.
\end{equation}
\end{theorem}

The proof follows immediately from Theorem~\ref{complete_reducibility_sl2}
and Theorem~\ref{primitive_dimensions}.

Thus, we have obtained the complete decomposition of the graph algebra
$\mathcal A$ into irreducible $\mathfrak{sl}_2$-modules.
As a by-product, passing to dimensions, we obtain the  binomial
identity
$$
\sum_{k=0}^{\lfloor m/2\rfloor}
\left(\binom{m}{k}-\binom{m}{k-1}\right)(m-2k+1)=2^m.
$$

\subsection{Involution}

Define the linear operator
\begin{equation}\label{involution}
\omega:\mathcal A\to\mathcal A,
\qquad
\omega(\boldsymbol{x}_S)=\boldsymbol{x}_{\overline S}.
\end{equation}
It is easy to check that the operator $\omega$ is a linear involution:
$
\omega^2=\mathrm{Id}_{\mathcal A},
$
and moreover
$
\omega(A_k)=A_{m-k}.
$

The following statement holds.
\begin{theorem}
The operators $D_+$ and $D_-$ are related by
$$
D_+=\omega D_-\omega.
$$
\end{theorem}

\begin{proof}
It is enough to check the statement on basis monomials $\boldsymbol{x}_S$.
We have
$$
D_-(\omega(\boldsymbol{x}_S))=D_-(\boldsymbol{x}_{\overline S})
=
\sum_{e\in \overline S}\boldsymbol{x}_{\overline S\setminus\{e\}}.
$$
Now apply $\omega$ to each summand:
$$
\omega\!\left(\boldsymbol{x}_{\overline S\setminus\{e\}}\right)
=
\boldsymbol{x}_{[m]\setminus(\overline S\setminus\{e\})}
=
\boldsymbol{x}_{S\cup\{e\}}.
$$
Therefore
$$
\omega D_-\omega(\boldsymbol{x}_S)
=
\sum_{e\in \overline S}\boldsymbol{x}_{S\cup\{e\}}
=
D_+(\boldsymbol{x}_S).
$$
Since both operators are linear, we obtain
$
D_+=\omega D_-\omega
$
on the whole of $\mathcal A$.
\end{proof}

Let
$$
\mathcal A^{\mathfrak{sl}_2}=\ker D_-\cap \ker D_+,
$$
be the vector space of $\mathfrak{sl}_2$-invariants, that is, the elements of
$\mathcal A$ which generate trivial $\mathfrak{sl}_2$-modules.

\begin{theorem}
The following statements hold.
\begin{enumerate}
\item[\textnormal{(i)}] The graph-complement involution preserves the space of
$\mathfrak{sl}_2$-invariants:
$$
\omega\bigl(\mathcal A^{\mathfrak{sl}_2}\bigr)
=
\mathcal A^{\mathfrak{sl}_2}.
$$

\item[\textnormal{(ii)}]
The following equality holds:
$$
\mathcal A^{\mathfrak{sl}_2}=
\begin{cases}
P_{m/2}, \quad \text{if $m$ is even},\\
0, \quad \text{if $m$ is odd}.
\end{cases}
$$
\end{enumerate}
\end{theorem}

\begin{proof}
If $f\in\mathcal A^{\mathfrak{sl}_2}$, then by the properties of the
involution we have
$$
D_+(\omega f)=\omega D_-f=0,
\qquad
D_-(\omega f)=\omega D_+f=0.
$$
Hence
$$
\omega f\in\mathcal A^{\mathfrak{sl}_2},
$$
which proves part \textnormal{(i)}.
To prove part \textnormal{(ii)}, note that $\mathfrak{sl}_2$-invariants have
weight zero. This is possible only under the condition
$
m-2j=0.
$
If $m$ is odd, such a $j$ does not exist, and therefore
$\mathcal A^\mathfrak{sl}_2=0$.
If $m$ is even, then all invariants lie in $A_{m/2}$, and therefore
$\mathcal A^{\mathfrak{sl}_2}=P_{m/2}$.
\end{proof}

\medskip

It is worth noting that the constructed $\mathfrak{sl}_2$-structure is a
graph-theoretic realization of the standard mechanism of raising and lowering
operators on Boolean and related graded posets. Unlike general rank-theoretic
applications of this mechanism, in our case the operators $D_+$ and $D_-$ have
a direct graph-theoretic interpretation as adding and deleting one edge.
In the next section we show that this structure is compatible with the action
of the pair group $\Gamma$ and therefore passes to the space of graph
invariants.


\section{The graph algebra $\mathcal A$ as a 
$\Gamma$-module}


In the previous section, the graph algebra $\mathcal A$ was considered as an
$\mathfrak{sl}_2$-module. We now consider the structure on $\mathcal A$
defined by the action of the pair group
$
\Gamma.
$
Unlike the $\mathfrak{sl}_2$-structure, where the decomposition of the algebra
$\mathcal A$ has a simple explicit form, the decomposition with respect to the
group
$
\Gamma
$
is much more complicated. In order to describe it, we first consider the
action of the larger symmetric group
$
S_m
$
on the set of variables $x_1,\dots,x_m$, and then restrict the obtained
$S_m$-modules to the subgroup $\Gamma\subseteq S_m$.

\subsection{Decomposition of $\mathcal A$ into irreducible 
$\Gamma$-modules}

Recall that the action of $\Gamma$ on the edge coordinate functions is given
by a permutation of indices and is extended multiplicatively and linearly to
the whole algebra $\mathcal A$.
Since the action of $\Gamma$ preserves the degree of a monomial, each rank
component
$
A_k
$
is a $\Gamma$-invariant subspace, and therefore it is natural to pose the
problem of decomposing these spaces into irreducible components.

For each $0\le k\le m$, the vector space $A_k$ is the permutation
$S_m$-module on the set of all $k$-element subsets of $[m]$.
If $0\le k\le \lfloor m/2\rfloor$, then this module can be identified with
the standard permutation module $M^{(m-k,k)}$, realized as the vector space
with basis given by $(m-k,k)$-tabloids, see~\cite{Sagan}.
For $k>\lfloor m/2\rfloor$, the corresponding description is obtained from
the previous one by taking complements of subsets, that is, via the natural
isomorphism $A_k\cong A_{m-k}$.

Recall that a tabloid may be viewed as a Young tableau in which the
requirement that the entries be ordered within rows is removed; a tabloid is
determined only by the sets of entries in each row, independently of their
order. If
$0\le k\le \lfloor m/2\rfloor$, then $(m-k,k)$ is a partition of $m$, and
such tabloids canonically correspond to $k$-element subsets of $[m]$: the
second row of the tabloid gives the subset $S$, while the first row gives its
complement $[m]\setminus S$. Thus, the identification
$$
\boldsymbol{x}_S
\longleftrightarrow
\left\{\begin{array}{c}
[m]\setminus S\\
S
\end{array}\right\},
$$
defines an isomorphism of $S_m$-modules
$$
A_k\cong M^{(m-k,k)},
\qquad 0\le k\le \left\lfloor \frac m2\right\rfloor.
$$

Irreducible $S_m$-modules over a field of characteristic zero are parametrized
by partitions of the number~$m$. The module corresponding to a partition
$\lambda \vdash m$ is called a \emph{Specht module} and is denoted by
$S^\lambda$.
In general, the decomposition of the permutation module $M^\lambda$ into
irreducible submodules is described by \emph{Young's rule}
\cite[Theorem~2.11.2]{Sagan}, which in our case simplifies to
$$
M^{(m-k,k)}
\cong
\bigoplus_{j=0}^{\min(k,m-k)}
S^{(m-j,j)}.
$$

The dimension of the Specht module $S^\lambda$ is equal to the number
$f^\lambda$ of standard Young tableaux of shape $\lambda$ and is computed by
the hook-length formula, see for example \cite[Section~4.3]{Sagan}.
For the two-row partition $(m-k,k)$, the hook-length formula, see
Fulton~\cite[Exercise~9]{Fulton1997}, after simplification gives
\begin{equation} \label{dim_S}
\dim S^{(m-k,k)}=f^{(m-k,k)}
=
\frac{m!(m-2k+1)}{(m-k+1)!k!}
=
\binom{m}{k}-\binom{m}{k-1}.
\end{equation}
As we see, the dimension of the Specht module $S^{(m-k,k)}$ turns out to be
the same as the dimension of the space of primitive elements $P_k$. We shall
return to this fact in the next subsection.

Each Specht module $S^{(m-k,k)}$ admits an explicit polynomial realization as
a subspace of $A_k$.
We briefly recall this construction; detailed definitions of Young diagrams
and Young tableaux can be found in~\cite{Sagan, FH, Fulton1997}.
A \emph{standard Young tableau} of shape~$\lambda$ is a filling of a Young
diagram by the numbers from $[m]$ such that the entries strictly increase in
each row from left to right and in each column from top to bottom.
To each such tableau~$T$ of shape $\lambda=(m-j,j)$ one associates the
polynomial $\Delta_T$, the product of Vandermonde differences over all
columns of the tableau.
The polynomials $\Delta_T$, where $T$ runs through all standard Young tableaux
of shape~$\lambda$, form a basis of the Specht module $S^\lambda$.

\begin{example}{\rm

Consider the partition
$
\lambda=(3,2).
$
In this case the standard tableaux have the form
$$
\young(123,45),
\qquad
\young(124,35),
\qquad
\young(125,34),
\qquad
\young(134,25),
\qquad
\young(135,24).
$$
and generate the following basis elements of the Specht module $S^{(3,2)}$:
\begin{align*}
&(x_1-x_4)(x_2-x_5), (x_1-x_3)(x_2-x_5),(x_1-x_3)(x_2-x_4),\\
&(x_1-x_2)(x_3-x_5), (x_1-x_2)(x_3-x_4).
\end{align*}
}
$\triangle$ \end{example}

\medskip

In what follows, these polynomial realizations of Specht modules will play an
important role in the description of $\mathfrak{sl}_2$-invariants and their
connection with orbit sums of graphs.

If $W$ is an $S_m$-module, then by
$
\operatorname{Res}^{S_m}_{\Gamma} W
$
we shall denote the same vector space, but with the action restricted from
$S_m$ to the subgroup $\Gamma\subseteq S_m$.
After restricting the action from $S_m$ to the subgroup $\Gamma$, we obtain
\begin{equation}\label{Ak_Gamma_restriction}
A_k
\cong_\Gamma
\operatorname{Res}^{S_m}_{\Gamma} M^{(m-k,k)}
\cong
\bigoplus_{j=0}^{\min(k,m-k)}
\operatorname{Res}^{S_m}_{\Gamma} S^{(m-j,j)}.
\end{equation}

Therefore, the problem of decomposing $A_k$ into irreducible $\Gamma$-modules
is reduced to the problem of decomposing the restrictions
$
\operatorname{Res}^{S_m}_{\Gamma} S^{(m-j,j)}
$
into irreducible $\Gamma$-modules, which we shall describe completely below.

Recall that for $n\ge 3$ the homomorphism
$
\varphi:S_n\longrightarrow \Gamma
$
is an isomorphism, see Subsection~\ref{ss2.1}. Therefore the theory of
finite-dimensional representations of the group $\Gamma$ over a field of
characteristic zero is identified with the representation theory of the
symmetric group $S_n$.

As in the general case, irreducible $S_n$-modules over a field of
characteristic zero are again parametrized by partitions
$
\lambda\vdash n
$
and are realized by the corresponding Specht modules. For each such partition
of the number $n$, denote by $U_\lambda$ the corresponding $\Gamma$-module
obtained by transporting the Specht module $S^\lambda$ through the
isomorphism $\varphi$. That is, as a vector space
$
U_\lambda=S^\lambda,
$
but the action of an element $\gamma\in\Gamma$ is given by
$$
\gamma\cdot u
=
\varphi^{-1}(\gamma)\cdot u,
\qquad
u\in S^\lambda.
$$

\begin{theorem}
Every irreducible $\Gamma$-module is isomorphic to exactly one module
$U_\lambda$, where $\lambda\vdash n$.
\end{theorem}

This statement is the direct transfer of the standard classification of
irreducible $S_n$-modules through the isomorphism $\varphi:S_n\cong\Gamma$.

\medskip

Irreducible $\Gamma$-modules admit an explicit polynomial realization
analogous to the realization of Specht modules given above.

\begin{example}
  For $n = 4$ we have five irreducible $\Gamma$-modules of dimensions
  $1,3,2,3,1$:
  $$
    U_{(4)},\quad
    U_{(3,1)},\quad
    U_{(2,2)},\quad
    U_{(2,1,1)},\quad
    U_{(1^4)}.
  $$
  Since these modules are associated with permutations of \emph{vertices},
  the natural coordinates are vertex variables, which we denote by
  $y_1,y_2,y_3,y_4$.
  In the polynomial realization through these variables one may take the
  following bases:
  $$
    U_{(4)}
    =
    \langle\,1\,\rangle,
  $$
  $$
    U_{(3,1)}
    =
    \langle\,
      y_1-y_2,\;
      y_1-y_3,\;
      y_1-y_4
    \,\rangle,
  $$
  $$
    U_{(2,2)}
    =
    \bigl\langle\,
      (y_1-y_3)(y_2-y_4),\;
      (y_1-y_2)(y_3-y_4)
    \,\bigr\rangle,
  $$
  $$
    U_{(2,1,1)}
    =
    \bigl\langle\,
      (y_1-y_3)(y_1-y_4)(y_3-y_4),\;
      (y_1-y_2)(y_1-y_4)(y_2-y_4),\;
      (y_1-y_2)(y_1-y_3)(y_2-y_3)
    \,\bigr\rangle,
  $$
  $$
    U_{(1^4)}
    =
    \left\langle\,
      \prod_{1\le i<j\le 4}(y_i-y_j)
    \,\right\rangle.
  $$
  Here $U_{(4)}$ is the trivial module, while $U_{(1^4)}$ is the sign module.

  $\triangle$
\end{example}

The polynomial realization above describes the type of an irreducible
$\Gamma$-module through vertex variables and is abstract: it does not yet
automatically define a concrete subspace in the graph algebra~$\mathcal{A}$,
which has edge coordinates
$x_1, x_2 \ldots, x_m$.

We can now decompose $A_k$ into irreducible $\Gamma$-modules.

Each summand $\operatorname{Res}^{S_m}_{\Gamma} S^{(m-j,j)}$ in
\eqref{Ak_Gamma_restriction} decomposes into irreducible $\Gamma$-modules:
$$
  \operatorname{Res}^{S_m}_{\Gamma} S^{(m-j,j)}
  \;\cong\;
  \bigoplus_{\lambda \,\vdash\, n}
  m_{\lambda,j}\, U_\lambda,
$$
where the multiplicities $m_{\lambda,j}$ are computed by the scalar product
of characters \cite[Chapter~2]{Sagan}:
$$
  m_{\lambda,j}
  \;=\;
  \bigl\langle
    \operatorname{Res}^{S_m}_{\Gamma}\chi^{(m-j,j)},\;
    \chi^\lambda
  \bigr\rangle_{\!\Gamma}
  \;=\;
  \frac{1}{n!}
  \sum_{\sigma \in S_n}
  \chi^{(m-j,j)}\!\bigl(\varphi(\sigma)\bigr)\cdot
  \chi^\lambda(\sigma).
$$

Substituting this into~\eqref{Ak_Gamma_restriction}, we obtain the full
decomposition of $A_k$ into irreducible $\Gamma$-modules:
\begin{equation}
  A_k
  \;\cong\;
  \bigoplus_{j=0}^{\min(k,\,m-k)}
  \bigoplus_{\lambda \,\vdash\, n}
  m_{\lambda,j}\, U_\lambda.
\end{equation}

\subsection{Primitive spaces and two-row Specht modules}

In this subsection we record an important connection between the primitive
spaces of the operator $D_-$ and two-row Specht modules. It will be used in
the next section when interpreting the Schur--Weyl decomposition.

In what follows, we shall need to consider not only the whole algebra
$\mathcal A$, but also its $\Gamma$-invariant part. Therefore we record that
the operators defining the $\mathfrak{sl}_2$-structure are compatible with
renumbering of vertices, that is, they are $\Gamma$-equivariant.

\begin{theorem}
The operators $D_+$, $D_-$, and $D_0$ commute with the action of the group
$\Gamma$. In particular, the spaces $A_k$ and $P_k$ are
$\Gamma$-invariant subspaces.
\end{theorem}

\begin{proof}
It is enough to check the statement on basis monomials
$\boldsymbol{x}_S$, $S\subseteq[m]$. Let $\gamma\in\Gamma$. Then
$$
\gamma\cdot \boldsymbol{x}_S=\boldsymbol{x}_{\gamma(S)}.
$$
We have
$$
D_+(\gamma\cdot \boldsymbol{x}_S)
=
D_+(\boldsymbol{x}_{\gamma(S)})
=
\sum_{e\notin \gamma(S)}
\boldsymbol{x}_{\gamma(S)\cup\{e\}}.
$$
Since $e\notin\gamma(S)$ if and only if $e=\gamma(f)$ for some
$f\notin S$, it follows that
$$
D_+(\gamma\cdot \boldsymbol{x}_S)
=
\sum_{f\notin S}
\boldsymbol{x}_{\gamma(S\cup\{f\})}
=
\gamma\cdot D_+(\boldsymbol{x}_S).
$$
Hence $D_+\gamma=\gamma D_+$.

Similarly,
$$
D_-(\gamma\cdot \boldsymbol{x}_S)
=
D_-(\boldsymbol{x}_{\gamma(S)})
=
\sum_{e\in\gamma(S)}
\boldsymbol{x}_{\gamma(S)\setminus\{e\}}
=
\sum_{f\in S}
\boldsymbol{x}_{\gamma(S\setminus\{f\})}
=
\gamma\cdot D_-(\boldsymbol{x}_S).
$$
Therefore $D_-\gamma=\gamma D_-$.

Finally, since $|\gamma(S)|=|S|$, the operator $D_0$, which depends only on
the number of edges, also commutes with the action of $\Gamma$. Hence
$D_+$, $D_-$, and $D_0$ are $\Gamma$-equivariant.

It follows immediately that $A_k$ and $P_k$ are $\Gamma$-invariant subspaces.
\end{proof}

Since $D_-$ commutes with the action of $S_m$, the kernel of $D_-$ is an
$S_m$-invariant subspace, that is, $P_k$ is an $S_m$-submodule of $A_k$.

\begin{theorem}
For each
$$
0\le k\le \left\lfloor\frac m2\right\rfloor,
$$
there is an isomorphism of $S_m$-modules
$$
P_k\cong S^{(m-k,k)}.
$$
\end{theorem}

\begin{proof}
We use the polynomial realization of two-row Specht modules constructed
above.
If $T$ is a standard Young tableau of shape $(m-k,k)$, and
$$
(\alpha_1,\beta_1),\dots,(\alpha_k,\beta_k),
$$
are its two-cell columns, then the corresponding Specht element has the form
$$
\Delta_T=\prod_{i=1}^{k}(x_{\alpha_i}-x_{\beta_i}).
$$
For each factor we have
$
D_-(x_{\alpha_i}-x_{\beta_i})=1-1=0.
$
Since the variables occurring in different factors are pairwise distinct, the
operator $D_-$ acts on such a product as a derivation, see~\eqref{Leibniz}.
Therefore
$$
D_-(\Delta_T)
=
\sum_{i=1}^{k}
\left(\prod_{\ell<i}(x_{\alpha_\ell}-x_{\beta_\ell})\right)
D_-(x_{\alpha_i}-x_{\beta_i})
\left(\prod_{\ell>i}(x_{\alpha_\ell}-x_{\beta_\ell})\right)
=0.
$$
Hence
$
\Delta_T\in P_k.
$

On the other hand, the elements $\Delta_T$, where $T$ runs through all
standard tableaux of shape $(m-k,k)$, form a basis of the Specht module
$S^{(m-k,k)}$. Therefore they generate in $P_k$ an $S_m$-submodule
isomorphic to $S^{(m-k,k)}$.
Since the constructed submodule has the same dimension, see \eqref{dim_S},
as $P_k$, it follows that
$
P_k\cong S^{(m-k,k)}
$
as an $S_m$-module.
\end{proof}

Thus, the primitive spaces $P_k$ realize precisely the two-row Specht modules.
In the next section, the Schur--Weyl decomposition will refine this fact and
explain why these Specht modules appear as multiplicity spaces for the
corresponding $\mathfrak{sl}_2$-chains.


\subsection{The case $n=4$: restriction of Specht modules to 
$\Gamma$}

Recall that the group $\Gamma$ is isomorphic to $S_4$, but it acts not on the
four vertices, but on the six edges of the complete graph $K_4$.

In the decomposition for $m=6$, four two-row Specht modules of the group
$S_6$ occur:
$$
S^{(6)},\qquad S^{(5,1)},\qquad S^{(4,2)},\qquad S^{(3,3)}.
$$
After restriction to the subgroup $\Gamma$, these Specht modules of the group
$S_6$ are, in general, no longer irreducible and decompose into irreducible
Specht modules of the group $\Gamma\simeq S_4$.

Denote the irreducible $\Gamma\simeq S_4$-modules by
$$
U_{(4)},\quad U_{(3,1)},\quad U_{(2,2)},\quad U_{(2,1,1)},\quad U_{(1^4)},
$$
indexed by all partitions of the number $4$.

Let us find their polynomial realization directly in the edge algebra
$\mathcal A$. For this purpose, introduce the incident edge functions
$$
\rho_i:=\sum_{j\ne i}x_{ij},
\qquad i=1,2,3,4.
$$
Taking into account the relabelling \eqref{perepoz}, we have
$$
\begin{aligned}
\rho_1&=x_{12}+x_{13}+x_{14}=x_1+x_2+x_3,\\
\rho_2&=x_{12}+x_{23}+x_{24}=x_1+x_4+x_5,\\
\rho_3&=x_{13}+x_{23}+x_{34}=x_2+x_4+x_6,\\
\rho_4&=x_{14}+x_{24}+x_{34}=x_3+x_5+x_6.
\end{aligned}
$$
The function $\rho_i$ is the degree of the vertex $i$ as a function on
graphs. If $\sigma\in S_4$, then the renumbering of vertices sends the edges
incident to vertex $i$ to the edges incident to vertex $\sigma(i)$. Therefore
$
\sigma\cdot \rho_i=\rho_{\sigma(i)}.
$
Hence the substitution of vertex variables
$
y_i\mapsto \rho_i
$
is $\Gamma$-equivariant. Consequently, the standard polynomial realizations
of Specht modules of the group $S_4$ in the variables $y_1,\dots,y_4$ give
realizations of the corresponding $\Gamma$-modules $U_\lambda$ in the edge
algebra $\mathcal A$.

In particular, we have the following natural copies of irreducible
$\Gamma$-modules in $\mathcal A$:
$$
U_{(4)}=\langle 1\rangle.
$$
The standard three-dimensional module is realized by differences of vertex
degrees:
$$
U_{(3,1)}
=
\langle
\rho_1-\rho_2,\ \rho_1-\rho_3,\ \rho_1-\rho_4
\rangle.
$$
Explicitly:
$$
\begin{aligned}
\rho_1-\rho_2&=(x_2+x_3)-(x_4+x_5),\\
\rho_1-\rho_3&=(x_1+x_3)-(x_4+x_6),\\
\rho_1-\rho_4&=(x_1+x_2)-(x_5+x_6).
\end{aligned}
$$

For the module $U_{(2,2)}$, one may take the basis
$$
U_{(2,2)}
=
\left\langle
(\rho_1-\rho_3)(\rho_2-\rho_4),\
(\rho_1-\rho_2)(\rho_3-\rho_4)
\right\rangle.
$$
This corresponds to the two standard tableaux of shape $(2,2)$. Note that we
have obtained one of the possible realizations of the abstract module
$U_{(2,2)}$, although it can also be realized in $A_1$, for example as
$$
U_{(2,2)}=\left\langle
x_1-x_2-x_5+x_6,\ 
x_1-x_3-x_4+x_6
\right\rangle.
$$

For the module $U_{(2,1,1)}$ we have the following polynomial realization:
$$
\begin{aligned}
U_{(2,1,1)}
=
\Bigl\langle\,
&(\rho_1-\rho_3)(\rho_1-\rho_4)(\rho_3-\rho_4),(\rho_1-\rho_2)(\rho_1-\rho_4)(\rho_2-\rho_4),(\rho_1-\rho_2)(\rho_1-\rho_3)(\rho_2-\rho_3)
\,\Bigr\rangle.
\end{aligned}
$$
Finally, the sign module is realized by the Vandermonde product
$$
U_{(1^4)}
=
\left\langle
\prod_{1\le i<j\le4}(\rho_i-\rho_j)
\right\rangle.
$$

\medskip

The characters of these modules on the conjugacy classes of $S_4$ have the
standard form, see for example \cite[\S 2.3]{FH}:
$$
\begin{array}{c|ccccc}
 & 1^4 & 2\,1^2 & 2^2 & 3\,1 & 4\\
\hline
U_{(4)}       & 1 & 1  & 1  & 1  & 1\\
U_{(3,1)}     & 3 & 1  & -1 & 0  & -1\\
U_{(2,2)}     & 2 & 0  & 2  & -1 & 0\\
U_{(2,1,1)}   & 3 & -1 & -1 & 0  & 1\\
U_{(1^4)}     & 1 & -1 & 1  & 1  & -1
\end{array}
$$
The sizes of the conjugacy classes are, respectively,
$$
1,\quad 6,\quad 3,\quad 8,\quad 6.
$$

We now need to know to which cycle types in $S_6$ these classes pass under
the action of $S_4$ on the edges of $K_4$. We have:
$$
\begin{array}{c|ccccc}
\text{type in }S_4      & 1^4 & 2\,1^2 & 2^2 & 3\,1 & 4\\
\hline
\text{type on edges in }S_6
& 1^6 & 1^2 2^2 & 1^2 2^2 & 3^2 & 4\,2
\end{array}
$$
For example, the transposition $(12)$ fixes the edges $12$ and $34$ and
interchanges in pairs the edges $13,23$ and $14,24$; hence its type on the
edges is $1^2 2^2$.
A double transposition has the same type on the edges. A 3-cycle gives two
orbits of length $3$, while a 4-cycle gives one orbit of length $4$ and one
orbit of length $2$.

The characters of the required Specht modules of $S_6$, restricted to these
classes, have the following form:
$$
\begin{array}{c|ccccc}
 & 1^4 & 2\,1^2 & 2^2 & 3\,1 & 4\\
\hline
\operatorname{Res}^{S_6}_{\Gamma}S^{(6)}
    & 1 & 1 & 1 & 1 & 1\\
\operatorname{Res}^{S_6}_{\Gamma}S^{(5,1)}
    & 5 & 1 & 1 & -1 & -1\\
\operatorname{Res}^{S_6}_{\Gamma}S^{(4,2)}
    & 9 & 1 & 1 & 0 & 1\\
\operatorname{Res}^{S_6}_{\Gamma}S^{(3,3)}
    & 5 & 1 & 1 & 2 & -1
\end{array}
$$
Let us briefly explain the origin of these values. For $S^{(6)}$, the
character is identically equal to $1$. For $S^{(5,1)}$, the standard formula
is used:
$$
\chi^{(5,1)}(\tau)=\#\operatorname{Fix}(\tau)-1,
$$
where $\tau\in S_6$. Therefore, on the types $1^6,1^2 2^2,3^2,4\,2$ we
obtain, respectively,
$$
5,\quad 1,\quad -1,\quad -1.
$$
For $S^{(4,2)}$, one may use the decomposition of the permutation module on
two-element subsets:
$$
M^{(4,2)}\cong S^{(6)}\oplus S^{(5,1)}\oplus S^{(4,2)}.
$$
The character of $M^{(4,2)}$ is equal to the number of two-element subsets
fixed by the permutation. Subtracting the characters of $S^{(6)}$ and
$S^{(5,1)}$ gives the row for $S^{(4,2)}$. Similarly,
$$
M^{(3,3)}
\cong
S^{(6)}\oplus S^{(5,1)}\oplus S^{(4,2)}\oplus S^{(3,3)},
$$
and the character of $M^{(3,3)}$ counts fixed three-element subsets. This
gives the row for $S^{(3,3)}$.

Now the decomposition of each restricted Specht module into irreducible
$\Gamma\simeq S_4$-modules is computed by the usual formula for the scalar
product of characters:
$$
\langle \chi,\psi\rangle_{S_4}
=
\frac1{24}\sum_{C}|C|\,\chi(C)\psi(C),
$$
where the sum is taken over the conjugacy classes of $S_4$.

Thus we obtain:
$$
\operatorname{Res}^{S_6}_{\Gamma}S^{(6)}
\cong
U_{(4)},
$$
$$
\operatorname{Res}^{S_6}_{\Gamma}S^{(5,1)}
\cong
U_{(3,1)}\oplus U_{(2,2)},
$$
$$
\operatorname{Res}^{S_6}_{\Gamma}S^{(4,2)}
\cong
U_{(4)}\oplus U_{(3,1)}\oplus U_{(2,2)}
\oplus U_{(2,1,1)},
$$
$$
\operatorname{Res}^{S_6}_{\Gamma}S^{(3,3)}
\cong
U_{(4)}\oplus U_{(3,1)}\oplus U_{(1^4)}.
$$
For example, for $S^{(4,2)}$ the character of the restriction is
$$
(9,1,1,0,1).
$$
Its scalar products with the rows of the character table of $S_4$ are
$$
1,\quad 1,\quad 1,\quad 1,\quad 0,
$$
which gives
$$
\operatorname{Res}^{S_6}_{\Gamma}S^{(4,2)}
\cong
U_{(4)}\oplus U_{(3,1)}\oplus U_{(2,2)}
\oplus U_{(2,1,1)}.
$$

Now we use the decomposition
$$
A_k\cong M^{(6-k,k)}
\cong
\bigoplus_{j=0}^{\min(k,6-k)}S^{(6-j,j)}
$$
and obtain the decompositions of the homogeneous components of the algebra
$\mathcal A$ as $\Gamma$-modules.

For $k=0$:
$$
A_0\cong U_{(4)}.
$$

For $k=1$:
$$
A_1\cong
U_{(4)}\oplus U_{(3,1)}\oplus U_{(2,2)}.
$$

For $k=2$:
$$
A_2\cong
2U_{(4)}
\oplus
2U_{(3,1)}
\oplus
2U_{(2,2)}
\oplus
U_{(2,1,1)}.
$$

For $k=3$:
$$
A_3\cong
3U_{(4)}
\oplus
3U_{(3,1)}
\oplus
2U_{(2,2)}
\oplus
U_{(2,1,1)}
\oplus
U_{(1^4)}.
$$

Taking into account the involution \eqref{involution}, we also have
$$
A_4\cong A_2,\qquad A_5\cong A_1,\qquad A_6\cong A_0.
$$

Therefore, the full decomposition of the algebra
$$
\mathcal A=\bigoplus_{k=0}^{6}A_k
$$
as a $\Gamma$-module has the form
$$
\mathcal A
\cong
11U_{(4)}
\oplus
9U_{(3,1)}
\oplus
8U_{(2,2)}
\oplus
3U_{(2,1,1)}
\oplus
U_{(1^4)}.
$$
Checking the dimension:
$$
11\cdot1+9\cdot3+8\cdot2+3\cdot3+1\cdot1
=
64
=
2^6
=
\dim\mathcal A.
$$

In particular, the multiplicity of the trivial module $U_{(4)}$ in
$\mathcal A$ is equal to $11$. This agrees with the fact that
$$
\dim \mathcal A^\Gamma=11,
$$
that is, with the number of non-isomorphic simple graphs on four vertices.


\subsection{Orbit sums and the invariant algebra 
$\mathcal A^\Gamma$}

The algebra of $\Gamma$-invariants is the subalgebra
$$
\mathcal A^\Gamma
=
\{f\in\mathcal A:\gamma\cdot f=f\ \text{for all }\gamma\in\Gamma\}.
$$
Since $\Gamma$ acts on $\mathcal A$ by algebra automorphisms,
$\mathcal A^\Gamma$ is a subalgebra of $\mathcal A$.

For a subset $S\subseteq[m]$, denote its $\Gamma$-orbit by
$$
\Gamma\cdot S=\{\gamma(S):\gamma\in\Gamma\}.
$$
The corresponding \emph{orbit sum} is the element
\begin{equation}
\mathbf O(\boldsymbol{x}_S):=\sum_{T\in\Gamma\cdot S}\boldsymbol{x}_T.
\end{equation}
If $G$ is a graph with edge set $S$, then $\mathbf O(\boldsymbol{x}_S)$ is
the sum of all monomials corresponding to graphs isomorphic to $G$.

\begin{theorem}
Let one representative $S$ be chosen from each $\Gamma$-orbit on the set of
all subsets of $[m]$. Then the orbit sums
$$
\mathbf O(\boldsymbol{x}_S),
$$
form a basis of the space $\mathcal A^\Gamma$.
\end{theorem}

The proof follows from a well-known fact from the invariant theory of finite
groups. For the action of $\Gamma$ on $\mathcal A$, the Reynolds operator is
well-defined:
$$
R_\Gamma(f):=\frac{1}{|\Gamma|}\sum_{\gamma\in\Gamma}\gamma\cdot f,
\qquad f\in\mathcal A.
$$
It is a linear projector from $\mathcal A$ onto the invariant subalgebra
$\mathcal A^\Gamma$,
see, for example, \cite{DK,Sturmfels}. The orbit sum
$\mathbf O(\boldsymbol{x}_S)$ differs from $R(\boldsymbol{x}_S)$ only by a
nonzero scalar factor. Therefore, the orbit sums taken over representatives
of $\Gamma$-orbits naturally give a basis of the invariant space
$\mathcal A^\Gamma$.

Since the action of $\Gamma$ preserves degree, this basis is compatible with
the rank structure:
$$
\mathcal A^\Gamma=\bigoplus_{k=0}^m A_k^\Gamma.
$$
The dimension $\dim A_k^\Gamma$ is equal to the number of non-isomorphic
simple graphs on $n$ vertices with $k$ edges.

\subsection{The ring structure of $\mathcal A^\Gamma$ and monogenicity}

We now consider $\mathcal A^\Gamma$ as a commutative algebra. From the
functional point of view, the structure of $\mathcal A^\Gamma$ is simple:
it is the algebra of all functions on the set of $\Gamma$-orbits.
Denote by
$$
\Omega_n:=\{0,1\}^m/\Gamma,
$$
the set of $\Gamma$-orbits on the Boolean cube. The elements of $\Omega_n$
correspond to isomorphism classes of simple graphs on $n$ vertices. Since the
elements of $\mathcal A^\Gamma$ are precisely those functions in
$\mathcal A\cong \operatorname{Fun}(\{0,1\}^m,\mathbb C)$ which are constant
on $\Gamma$-orbits, we have a natural isomorphism of algebras
\begin{equation}
\mathcal A^\Gamma\cong \operatorname{Fun}(\Omega_n,\mathbb C).
\end{equation}

Let
$
\Omega_n=\{\omega_1,\dots,\omega_N\},
N=|\Omega_n|.
$
Then
$$
\operatorname{Fun}(\Omega_n,\mathbb C)\cong \mathbb C^N,
$$
as a commutative $\mathbb C$-algebra.
Thus, the ring structure of $\mathcal A^\Gamma$ is the structure of a direct
product of a finite number of copies of the field $\mathbb C$.
Therefore, just like the algebra $\mathcal A$, the algebra
$\mathcal A^\Gamma$ is monogenic.

\begin{theorem}
The algebra $\mathcal A^\Gamma$ is monogenic as an abstract
$\mathbb C$-algebra:
$$
\mathcal A^\Gamma=\mathbb C[z],
$$
for some $z\in\mathcal A^\Gamma$.
\end{theorem}

This theorem can also be proved constructively, through Lagrange
interpolation, analogously to the proof of Theorem~\ref{A_monogenic};
however, the generating elements obtained in this way are cumbersome. For
practical computations, it is more convenient to use the fact that the
generating element takes distinct values on $\Gamma$-orbits.

\begin{example}{\rm
Let $n=4$. The number of $\Gamma$-orbits on the set of all subsets of edges
is equal to the number of non-isomorphic simple graphs on four vertices:
$$
\dim \mathcal A^\Gamma=11.
$$

Consider three natural orbital invariants:
$
\mathbf O({x}_1),
$
which counts the number of edges of a graph,
$
\mathbf O({x}_1{x}_2),
$
where ${x}_1{x}_2$ corresponds to two adjacent edges, and therefore it counts
the number of two-edge subgraphs, and
$
T=\mathbf O({x}_1{x}_2{x}_4),
$
where ${x}_1{x}_2{x}_4$ corresponds to a triangle, and it counts the number
of triangles in a graph.

Put
$$
z=\mathbf O({x}_1)
+
\mathbf O({x}_1{x}_2)
+
\mathbf O({x}_1{x}_2{x}_4)\in\mathcal A^\Gamma.
$$
We show that this element takes pairwise distinct values on all $11$
isomorphism classes of graphs on four vertices. For this it is enough to
compute the values of the triple of invariants
$$
(e(G),p(G),t(G)),
$$
where $e(G)$ is the number of edges of $G$, $p(G)$ is the number of pairs of
adjacent edges, and $t(G)$ is the number of triangles.

$$
\begin{array}{c|c|c|c|c}
\text{graph type} & e(G) & p(G) & t(G) & z(G)=e(G)+p(G)+t(G)\\
\hline
\varnothing & 0 & 0 & 0 & 0\\
K_2\sqcup 2K_1 & 1 & 0 & 0 & 1\\
2K_2 & 2 & 0 & 0 & 2\\
P_3\sqcup K_1 & 2 & 1 & 0 & 3\\
P_4 & 3 & 2 & 0 & 5\\
K_{1,3} & 3 & 3 & 0 & 6\\
K_3\sqcup K_1 & 3 & 3 & 1 & 7\\
C_4 & 4 & 4 & 0 & 8\\
K_3\text{ with a pendant edge} & 4 & 5 & 1 & 10\\
K_4-e & 5 & 8 & 2 & 15\\
K_4 & 6 & 12 & 4 & 22
\end{array}
$$

Hence the values of the element $z$ on all orbits are pairwise distinct.
Therefore, in the case $n=4$, the invariant algebra $\mathcal A^\Gamma$ is
generated by one element
$$
z=
\mathbf O({x}_1)
+
\mathbf O({x}_1{x}_2)
+
\mathbf O({x}_1{x}_2{x}_4).
$$
}
$\triangle$ \end{example}

\subsection{The invariant algebra from the $\Gamma$-module decomposition}

Let us return to the decomposition of the homogeneous components $A_k$ as
$\Gamma$-modules. It gives another, representation-theoretic, description of
the invariant algebra $\mathcal A^\Gamma$.

Indeed, since
$$
\mathcal A^\Gamma=\bigoplus_{k=0}^{m}A_k^\Gamma,
$$
it is enough to understand the spaces $A_k^\Gamma$. From the decomposition
$$
A_k
\cong
\bigoplus_{j=0}^{\min(k,m-k)}
\operatorname{Res}^{S_m}_{\Gamma}S^{(m-j,j)},
$$
we obtain
$$
A_k^\Gamma
\cong
\bigoplus_{j=0}^{\min(k,m-k)}
\left(
\operatorname{Res}^{S_m}_{\Gamma}S^{(m-j,j)}
\right)^\Gamma.
$$
In other words, the space of invariants $A_k^\Gamma$ is formed precisely
from those components of the restrictions of two-row Specht modules which
contain the trivial $\Gamma$-module.

Denote
$$
a_k
=
\dim
\left(
\operatorname{Res}^{S_m}_{\Gamma}S^{(m-k,k)}
\right)^\Gamma = \dim( P_k^\Gamma), \quad  b_k=\dim A_k^\Gamma.
$$
Then
\begin{equation}\label{dim_AG}
b_k=
\sum_{j=0}^{\min(k,m-k)}a_j.
\end{equation}

Thus, the number of non-isomorphic graphs with $k$ edges can be read not only
as the number of $\Gamma$-orbits on $k$-element subsets of edges, but also as
the sum of the multiplicities of the trivial $\Gamma$-module in the
restrictions of two-row Specht modules.

This gives two complementary descriptions of the same space: the orbital
description
$$
A_k^\Gamma
=
\left\langle
\mathbf O(\boldsymbol{x}_S): |S|=k
\right\rangle,
$$
and the representation-theoretic description
$$
A_k^\Gamma
\cong
\bigoplus_{j=0}^{\min(k,m-k)}
\left(
\operatorname{Res}^{S_m}_{\Gamma}S^{(m-j,j)}
\right)^\Gamma.
$$
The first description is indexed by non-isomorphic graphs with $k$ edges,
whereas the second is indexed by the trivial components in the corresponding
restrictions of Specht modules.

\begin{example}{\rm
In the case $n=4$ we have $m=6$. From the decomposition already obtained
$$
A_0\cong U_{(4)},
$$
$$
A_1\cong
U_{(4)}\oplus U_{(3,1)}\oplus U_{(2,2)},
$$
$$
A_2\cong
2U_{(4)}
\oplus
2U_{(3,1)}
\oplus
2U_{(2,2)}
\oplus
U_{(2,1,1)},
$$
$$
A_3\cong
3U_{(4)}
\oplus
3U_{(3,1)}
\oplus
2U_{(2,2)}
\oplus
U_{(2,1,1)}
\oplus
U_{(1^4)},
$$
we immediately read off the dimensions of the invariant parts:
$$
\dim A_0^\Gamma=1,\qquad
\dim A_1^\Gamma=1,\qquad
\dim A_2^\Gamma=2,\qquad
\dim A_3^\Gamma=3.
$$
Indeed, in each $A_k$ the invariant part is contributed only by copies of the
trivial module $U_{(4)}$.

From the existence of the involution \eqref{involution} we obtain the
isomorphism
$
A_k\cong A_{6-k},
$
and therefore
$$
(\dim A_0^\Gamma,\dim A_1^\Gamma,\dots,\dim A_6^\Gamma)
=
(1,1,2,3,2,1,1).
$$
This is precisely the sequence of the numbers of non-isomorphic simple graphs
on four vertices by the number of edges. Summing up, we obtain
$$
\dim\mathcal A^\Gamma
=
1+1+2+3+2+1+1
=
11,
$$
which coincides with the number of non-isomorphic simple graphs on four
vertices.
}
$\triangle$
\end{example}

Thus, orbit sums give a natural combinatorial basis of $\mathcal A^\Gamma$,
whereas the decomposition of $A_k$ into $\Gamma$-modules explains from which
representation-theoretic components these invariants arise. In particular,
$\mathcal A^\Gamma$ is the trivial isotypic component of the graph algebra
with respect to the action of the group $\Gamma$.

\section{Constructive Schur--Weyl decomposition and graph orbit enumeration}

In the previous sections, the graph algebra $\mathcal A$ was described from
two points of view: through the action of the pair group $\Gamma$, which leads
to orbital graph invariants, and through the operators $D_+,D_0,D_-$, which
define an $\mathfrak{sl}_2$-structure. In this section we combine these two
descriptions by means of a specialization of the Schur--Weyl decomposition.

\subsection{The isomorphism $\mathcal A\cong V_1^{\otimes m}$}

The $\mathfrak{sl}_2$-action constructed above on the graph algebra
$\mathcal A$ has a very natural tensor interpretation.

Recall that $V_1$ denotes the standard irreducible $\mathfrak{sl}_2$-module
of dimension $2$ with basis
$
v_0,\ v_1
$
on which the action of the operators $e_+,e_-,h$ is given by the formulas
\begin{equation}\label{V1_action}
\begin{gathered}
e_+v_0=v_1,\qquad e_+v_1=0,\\
e_-v_1=v_0,\qquad e_-v_0=0,
\\
hv_0=-v_0,\qquad hv_1=v_1.
\end{gathered}
\end{equation}
Thus, $v_0$ has weight $-1$, while $v_1$ has weight $1$.

Now consider the tensor power
$
V_1^{\otimes m}
$
on which the algebra $\mathfrak{sl}_2$ acts diagonally on all tensor factors.
The space $V_1^{\otimes m}$ has a natural tensor basis
$$
v_{\varepsilon_1}\otimes v_{\varepsilon_2}\otimes\cdots\otimes v_{\varepsilon_m},
\qquad
\varepsilon_i\in\{0,1\}.
$$
The basis vectors of this space are indexed by binary words of length $m$,
or, equivalently, by subsets of the set $[m]$.
Indeed, to each subset $S\subseteq[m]$ we assign the tensor basis vector
\begin{equation}\label{ tensor_basis_S}
v_S:=w_1\otimes w_2\otimes\cdots\otimes w_m,
\qquad
w_i=
\begin{cases}
v_1,& i\in S,\\
v_0,& i\notin S.
\end{cases}
\end{equation}
Thus, the presence of the vector $v_1$ in a tensor factor means that the
corresponding edge is ``chosen'', while $v_0$ means that it is ``not chosen''.
Therefore there arises a natural linear map
$$
\Phi:V_1^{\otimes m}\longrightarrow \mathcal A,
$$
defined on basis vectors by the formula
$$
\Phi(v_S)=\boldsymbol{x}_S,
\qquad
S\subseteq[m].
$$
Since both $\{v_S:\ S\subseteq[m]\}$ and
$\{\boldsymbol{x}_S:\ S\subseteq[m]\}$ are bases of the corresponding
spaces, the map $\Phi$ is an isomorphism of vector spaces.
This isomorphism may be viewed as an algebraic formulation of the obvious
combinatorial fact that each of the $m$ possible edges of a graph is either
present or absent.
Thus, the space of all graphs on $n$ vertices is a product of $m$ two-state
systems.

For us it is important that the $\mathfrak{sl}_2$-action realizes elementary
operations of changing these states, that is, that this isomorphism is
compatible with the action of $\mathfrak{sl}_2$. The following theorem holds.

\begin{theorem}\label{A_tensor_power}
The isomorphism $\Phi$ is an isomorphism of $\mathfrak{sl}_2$-modules.
\end{theorem}

\begin{proof}
Since $\Phi$ is already an isomorphism of vector spaces, it is enough to
check that it commutes with the action of the three operators $e_+,e_-,h$.
Let $S\subseteq[m]$, and consider the corresponding basis vector $v_S$ from
\eqref{ tensor_basis_S}.
First compute the action of the operator $e_+$ on $v_S$.
The operator $e_+$ acts on a tensor product as a derivation, and, according
to \eqref{V1_action}, we have
$$
e_+v_0=v_1,\qquad e_+v_1=0.
$$
Hence a nonzero contribution appears only in those tensor factors for which
the corresponding index does not belong to $S$.
Therefore
$$
e_+v_S=\sum_{e\in \overline S} v_{S\cup\{e\}},
\qquad
\overline S=[m]\setminus S.
$$
Applying $\Phi$, we obtain
$$
\Phi(e_+v_S)=\sum_{e\in\overline S}\boldsymbol{x}_{S\cup\{e\}}=D_+(\boldsymbol{x}_S)=D_+(\Phi(v_S)).
$$
Thus,
$$
\Phi(e_+v_S)=D_+(\Phi(v_S)).
$$

Similarly, the operator $e_-$ replaces $v_1$ by $v_0$ in all possible
positions where this replacement does not give zero.
Therefore
$$
e_-v_S=\sum_{e\in S}v_{S\setminus\{e\}}.
$$
After applying $\Phi$, we have
$$
\Phi(e_-v_S)=\sum_{e\in S}\boldsymbol{x}_{S\setminus\{e\}}=D_-(\boldsymbol{x}_S)=D_-(\Phi(v_S)).
$$
Commutativity with the operator $h$ follows automatically from the already
proved commutativity with the operators $e_-$, $e_+$ and from the identity
$[e_+, e_-]=h$.
Thus, for all $S\subseteq[m]$ we have
$$
\Phi(e_+v_S)=D_+(\Phi(v_S)),\qquad
\Phi(e_-v_S)=D_-(\Phi(v_S)),\qquad
\Phi(hv_S)=D_0(\Phi(v_S)).
$$
Since the basis vectors $v_S$ generate the whole space $V_1^{\otimes m}$,
these equalities hold on the entire space.
Therefore $\Phi$ is an isomorphism of $\mathfrak{sl}_2$-modules.
\end{proof}

Theorem~\ref{A_tensor_power} gives a conceptual interpretation of the
graph algebra as a tensor $\mathfrak{sl}_2$-module. To each edge of the
complete graph $K_n$ there corresponds one copy of the standard two-dimensional
module $V_1$, and the whole algebra $\mathcal A$ is identified with the
tensor product of these elementary two-state modules. At the same time, the
isomorphism $\Phi$ is compatible with the natural action of the group $S_m$:
a permutation of the tensor factors in $V_1^{\otimes m}$ corresponds to a
permutation of the edge coordinates $x_1,\dots,x_m$ in the algebra
$\mathcal A$. Thus, the graph algebra appears as a tensor model in which
$\mathfrak{sl}_2$ and the symmetric group $S_m$ act simultaneously. In this
sense it is a natural analogue of the classical situation in the invariant
theory of binary forms, where $\mathfrak{sl}_2$-modules arise from symmetric
powers of the standard module $V_1$.

\subsection{Schur--Weyl duality in the graph algebra setting}

Schur--Weyl duality is one of the basic principles of representation theory,
describing the interaction of two natural actions on a tensor power of a
vector space. Let $V$ be a finite-dimensional vector space over a field of
characteristic zero. On the tensor power
$
V^{\otimes m}
$
there are two symmetries. First, the group $S_m$ acts by permuting tensor
factors:
$$
\sigma\cdot(v_1\otimes\cdots\otimes v_m)
=
v_{\sigma^{-1}(1)}\otimes\cdots\otimes v_{\sigma^{-1}(m)}.
$$
Second, the group $GL(V)$ acts diagonally:
$$
g\cdot(v_1\otimes\cdots\otimes v_m)
=
gv_1\otimes\cdots\otimes gv_m.
$$
These two actions commute.
Schur--Weyl duality~\cite[Chapter~6]{FH} states that each of these two
actions is the centralizer of the other in $\mathrm{End}(V_1^{\otimes m})$.
As a consequence, one obtains the decomposition
$$
  V_1^{\otimes m}
  \cong
  \bigoplus_{\substack{\lambda \vdash m \\ \ell(\lambda) \le 2}}
  S^\lambda \otimes L_\lambda,
$$
where $S^\lambda$ is an irreducible $S_m$-module, that is, a Specht module,
and $L_\lambda$ is the irreducible polynomial $GL_2$-module of highest weight
$\lambda$ \cite[Theorem~9.1.2]{GW}.

We apply this classical Schur--Weyl duality to our graph algebra.

\begin{theorem}
  
  As an $\mathfrak{sl}_2 \times S_m$-module, the graph algebra has the
  following decomposition:
  \begin{equation}\label{ SW}
    \mathcal{A}
    \cong
    \bigoplus_{k=0}^{\lfloor m/2 \rfloor}
    S^{(m-k,k)} \otimes V_{m-2k},
  \end{equation}
  where $S^{(m-k,k)}$ is the Specht module of the group~$S_m$, and
  $V_{m-2k}$ is the standard irreducible $\mathfrak{sl}_2$-module of highest
  weight $m-2k$.
\end{theorem}

\begin{proof}
  By the isomorphism of $\mathfrak{sl}_2$-modules
  $\mathcal{A} \cong V_1^{\otimes m}$, it is enough to apply Schur--Weyl
  duality to $V_1^{\otimes m}$.
  For the two-row partition $\lambda = (m-k, k)$, the irreducible
  $GL_2$-module of highest weight $\lambda$ has the form
  $$
    L_{(m-k,k)}
    \cong
    \operatorname{Sym}^{m-2k}(V_1) \otimes (\det)^k
    \quad\text{as a } GL_2\text{-module;}
  $$
  after restriction to the subgroup $SL_2$, where $\det = 1$, we obtain
  $L_{(m-k,k)}\big|_{SL_2} \cong \operatorname{Sym}^{m-2k}(V_1) = V_{m-2k}$.
  Substituting this into the Schur--Weyl decomposition and returning through
  the isomorphism $\mathcal{A} \cong V_1^{\otimes m}$, we obtain~\eqref{ SW}.
\end{proof}

If in~\eqref{ SW} we forget the action of $S_m$, then the factor
$S^{(m-k,k)}$ remains only as a multiplicity space. Hence, as an
$\mathfrak{sl}_2$-module,
\begin{equation}
  \mathcal A
  \cong
  \bigoplus_{k=0}^{\lfloor m/2 \rfloor}
\dim S^{(m-k,k)}\cdot V_{m-2k}=\bigoplus_{k=0}^{\lfloor m/2 \rfloor}
\left( \binom{m}{k}-\binom{m}{k-1} \right) \cdot V_{m-2k},
\end{equation}
and we again obtain the result of Theorem~\ref{irred_decomp_A}.

If, conversely, we forget the $\mathfrak{sl}_2$-action, then as an
$S_m$-module we have
$$
  \mathcal A
  \cong
  \bigoplus_{k=0}^{\lfloor m/2 \rfloor}
   (m-2k+1)\,S^{(m-k,k)}.
$$
However, in order to describe the space of graph invariants, it is important
to retain the $\mathfrak{sl}_2$-structure as well. Since the action of the
group $\Gamma$ commutes with the action of $\mathfrak{sl}_2$, passing to
$\Gamma$-invariants in the Schur--Weyl decomposition gives an isomorphism of
$\mathfrak{sl}_2$-modules
\begin{equation}\label{ab}
  \mathcal A^\Gamma
  \cong
  \bigoplus_{k=0}^{\lfloor m/2 \rfloor}
  \left(
  \operatorname{Res}^{S_m}_{\Gamma}S^{(m-k,k)}
  \right)^\Gamma
  \otimes V_{m-2k}.
\end{equation}

Passing to dimensions, we obtain
$$
\dim \mathcal A^\Gamma
=
\sum_{k=0}^{\lfloor m/2 \rfloor}
(m-2k+1)
\dim
\left(
\operatorname{Res}^{S_m}_{\Gamma}S^{(m-k,k)}
\right)^\Gamma.
$$

\medskip

Thus, in our case, the role of Schur--Weyl duality is that it explains the
simultaneous presence of three structures: the $\mathfrak{sl}_2$ operators,
the two-row Specht modules of the symmetric group $S_m$, and the graph
invariants with respect to the subgroup $\Gamma$.

\subsection{Orbit-sum coordinates and primitive decomposition}

In the space of invariants $A_r^\Gamma$ two descriptions naturally coexist.
The first is combinatorial: it is given by orbit sums and is therefore indexed
by non-isomorphic graphs with $r$ edges. The second is representation-theoretic:
it is given by primitive components with respect to the operator $D_-$ and
their liftings by the operator $D_+$. The aim of this subsection is to explain
how these two descriptions are related to each other.

The following theorem shows that every invariant from $A_r^\Gamma$ uniquely
decomposes into contributions of primitive $\Gamma$-invariants from lower
subspaces, lifted by the operator $D_+$.

\begin{theorem}\label{ AkGamma_primitive_decomposition}
For each $r=0,1,\dots,m$ one has the decomposition
$$
A_r^\Gamma
=
\bigoplus_{j=0}^{\min(r,m-r)}
D_+^{\,r-j}P_j^\Gamma.
$$
\end{theorem}

\begin{proof}
The decomposition \eqref{P_to_A} is compatible with the action of the group
$\Gamma$. Indeed, the operator $D_+$ commutes with the action of $\Gamma$,
and the spaces $P_j$ are $\Gamma$-invariant, since $D_-$ also commutes with
the action of $\Gamma$. Hence each summand $D_+^{\,r-j}P_j$ in the
decomposition
$$
A_r
=
\bigoplus_{j=0}^{\min(r,m-r)}
D_+^{\,r-j}P_j,
$$
is a $\Gamma$-submodule.
Passing to $\Gamma$-invariants, we obtain
$$
A_r^\Gamma
=
\bigoplus_{j=0}^{\min(r,m-r)}
\left(D_+^{\,r-j}P_j\right)^\Gamma.
$$
It remains to note that
$$
\left(D_+^{\,r-j}P_j\right)^\Gamma
=
D_+^{\,r-j}P_j^\Gamma.
$$
Indeed, the operator $D_+^{\,r-j}$ commutes with the action of the group
$\Gamma$.
If
$$
u=D_+^{\,r-j}p\in \left(D_+^{\,r-j}P_j\right)^\Gamma,
$$
then, applying the Reynolds operator $R_\Gamma$, we have
$$
u
=
R_\Gamma(u)
=
R_\Gamma(D_+^{\,r-j}p)
=
D_+^{\,r-j}R_\Gamma(p).
$$
Since the space $P_j$ is $\Gamma$-invariant, we have
$R_\Gamma(p)\in P_j^\Gamma$. Hence
$$
u\in D_+^{\,r-j}P_j^\Gamma.
$$
The reverse inclusion follows immediately from the fact that $D_+^{\,r-j}$
commutes with the action of $\Gamma$.
\end{proof}

From Theorem~\ref{ AkGamma_primitive_decomposition} formula
\eqref{dim_AG} follows immediately and acquires an additional meaning:
the invariants in $A_r^\Gamma$ consist of liftings of primitive
$\Gamma$-invariants from lower subspaces along $\mathfrak{sl}_2$-chains.

In the next subsection we describe a constructive method for building the
spaces $P_j$ and $P_j^\Gamma$ through bases of Specht modules.

\subsection{Primitive Specht vectors and orbit-sum coordinates}

Recall that the polynomial basis of the Specht module is generated by the
elements
$$
\Delta_T
=
\prod_{i=1}^{k}(x_{\alpha_i}-x_{\beta_i}),
$$
where $T$ runs over standard Young tableaux of shape $(m-k,k)$,
$0\le k\le \lfloor m/2\rfloor$, and
$$
(\alpha_1,\beta_1),\dots,(\alpha_k,\beta_k),
$$
are its two-cell columns; these elements form a basis of the space $P_k$.

In the Schur--Weyl interpretation, the space $P_k$ is the multiplicity space
for $\mathfrak{sl}_2$-chains of type $V_{m-2k}$, and the module generated by
it has the form
$$
\bigoplus_{s=0}^{m-2k}D_+^sP_k
\cong
S^{(m-k,k)}\otimes V_{m-2k}.
$$
Therefore the elements $\Delta_T$ may be regarded as initial vectors of these
$\mathfrak{sl}_2$-chains.

We now pass to the $\Gamma$-invariant part. Since the Reynolds operator
$R_\Gamma$ commutes with $D_-$, it maps $P_k$ into $P_k^\Gamma$. Hence, for
each standard tableau $T$ of shape $(m-k,k)$, we have
$$
R_\Gamma(\Delta_T)\in P_k^\Gamma.
$$
Moreover,
$$
P_k^\Gamma
=
\operatorname{span}
\left\{
R_\Gamma(\Delta_T):
T\ \text{a standard tableau of shape }(m-k,k)
\right\}.
$$
After discarding linearly dependent elements, this system gives a basis of
the space $P_k^\Gamma$.

Let us now express these elements in coordinates with respect to the basis
formed by orbit sums. Expanding the product in $\Delta_T$, we obtain
$$
\Delta_T
=
\sum_{\varepsilon\in\{0,1\}^{k}}
(-1)^{|\varepsilon|}
\boldsymbol{x}_{S_\varepsilon},
$$
where
$$
S_\varepsilon
=
\{\alpha_i:\varepsilon_i=0\}
\cup
\{\beta_i:\varepsilon_i=1\}.
$$
Thus, each monomial in the expansion of $\Delta_T$ corresponds to choosing
exactly one element from each two-cell column of the tableau $T$.

Applying the Reynolds operator, we have
$$
R_\Gamma(\Delta_T)
=
\sum_{\varepsilon\in\{0,1\}^{k}}
(-1)^{|\varepsilon|}
R_\Gamma(\boldsymbol{x}_{S_\varepsilon}).
$$
For what follows, it is convenient to use non-normalized group sums
$$
\mathcal R_S
:=
\sum_{\gamma\in\Gamma}\gamma\cdot\boldsymbol{x}_S
=
|\Gamma|\,R_\Gamma(\boldsymbol{x}_S)
=
|\Gamma_S|\,\mathbf O(\boldsymbol{x}_S),
$$
where $\Gamma_S$ is the stabilizer of the set $S$. Then
$$
|\Gamma|R_\Gamma(\Delta_T)
=
\sum_{\varepsilon\in\{0,1\}^{k}}
(-1)^{|\varepsilon|}
\mathcal R_{S_\varepsilon}.
$$

Let
$$
[G_1],\dots,[G_{b_k}],
$$
be all $b_k$ isomorphism classes of graphs with $k$ edges, and let
$\mathcal R_{G_j}$ be the corresponding group sums taken for fixed
representatives of these classes. Then
$$
|\Gamma|R_\Gamma(\Delta_T)
=
\sum_{j=1}^{b_k} c_{j,T}\mathcal R_{G_j},
$$
where
$$
c_{j,T}
=
\sum_{\substack{\varepsilon\in\{0,1\}^{k}\\
G_{S_\varepsilon}\cong G_j}}
(-1)^{|\varepsilon|}.
$$
Thus, $c_{j,T}$ is the signed number of choices of one element from each
two-cell column of the tableau $T$ which give a graph of isomorphism type
$G_j$.

Therefore, the preceding considerations lead to the following statement.

\begin{theorem}
Let $C_k=(c_{j,T})$ be the matrix whose rows are indexed by isomorphism
classes of graphs with $k$ edges and whose columns are indexed by standard
tableaux of shape $(m-k,k)$. Then the column space of the matrix $C_k$
coincides with the coordinate image of the space $P_k^\Gamma$ in the basis of
group orbit sums $\mathcal R_{G_j}$.
\end{theorem}

Thus, the matrix \(C_k\) is an explicit transition matrix from the Specht
basis of the primitive components \(P_k\) to the orbital description of graph
invariants:
$$
\Delta_T
\quad\longmapsto\quad
R_\Gamma(\Delta_T)
\quad\longmapsto\quad
(c_{1T},\dots,c_{b_kT})^{\mathsf T}.
$$
In these coordinates, primitive invariants are sign-changing linear
combinations of orbit sums of graphs with \(k\) edges.

After applying the operator \(D_+^{\,r-k}\) to such elements, we obtain
invariants in higher subspaces:
$$
D_+^{\,r-k}R_\Gamma(\Delta_T)
\in A_r^\Gamma,
\qquad k\le r\le m-k.
$$
Thus, the whole \(\mathfrak{sl}_2\)-chain structure in
\(\mathcal A^\Gamma\) can be described through liftings of signed orbital
combinations obtained from the Specht basis.

\begin{example}{\rm
Consider the case \(n=4,k=2\). There are two isomorphism classes of graphs
with two edges:
$$
G_A=\text{two adjacent edges},\qquad
G_D=\text{two independent edges}.
$$
Therefore the space \(A_2^\Gamma\) has orbital coordinates with respect to
the two group sums
$
\mathcal R_{G_A}, \mathcal R_{G_D}.
$

There are \(9\) standard Young tableaux of shape \((4,2)\):
$$
f^{(4,2)}=\binom62-\binom61=15-6=9.
$$
Denote them by \(T_1,\dots,T_9\):
$$
\begin{array}{ccc}
T_1=
\begin{array}{cccc}
1&3&5&6\\
2&4
\end{array}
&
T_2=
\begin{array}{cccc}
1&3&4&6\\
2&5
\end{array}
&
T_3=
\begin{array}{cccc}
1&3&4&5\\
2&6
\end{array}
\\[12pt]
T_4=
\begin{array}{cccc}
1&2&5&6\\
3&4
\end{array}
&
T_5=
\begin{array}{cccc}
1&2&4&6\\
3&5
\end{array}
&
T_6=
\begin{array}{cccc}
1&2&4&5\\
3&6
\end{array}
\\[12pt]
T_7=
\begin{array}{cccc}
1&2&3&6\\
4&5
\end{array}
&
T_8=
\begin{array}{cccc}
1&2&3&5\\
4&6
\end{array}
&
T_9=
\begin{array}{cccc}
1&2&3&4\\
5&6
\end{array}
\end{array}
$$

For each tableau \(T_i\), we take only its two-cell columns and form the
Specht vector
$$
\Delta_{T_i}=(x_{\alpha_1}-x_{\beta_1})(x_{\alpha_2}-x_{\beta_2}).
$$
After expanding the product and grouping the monomials according to the two
isomorphism types of graphs, we obtain decompositions
$$
|\Gamma|R_\Gamma(\Delta_{T_i})
=
c_{A,i}\mathcal R_{G_A}
+
c_{D,i}\mathcal R_{G_D}.
$$
A direct count gives:
$$
\begin{array}{c|c|c|c}
i & \Delta_{T_i} & c_{A,i} & c_{D,i}\\
\hline
1 &(x_1-x_2)(x_3-x_4) & 0 & 0\\
2 &(x_1-x_2)(x_3-x_5) & -1 & 1\\
3 &(x_1-x_2)(x_3-x_6) & 1 & -1\\
4 &(x_1-x_3)(x_2-x_4) & -1 & 1\\
5 &(x_1-x_3)(x_2-x_5) & 0 & 0\\
6 &(x_1-x_3)(x_2-x_6) & 1 & -1\\
7 &(x_1-x_4)(x_2-x_5) & 0 & 0\\
8 &(x_1-x_4)(x_2-x_6) & 1 & -1\\
9 &(x_1-x_5)(x_2-x_6) & 2 & -2
\end{array}
$$
Thus, if the rows of the matrix \(C_2\) are ordered as
$$
G_A,\ G_D,
$$
and the columns as
$$
T_1,\dots,T_9,
$$
then
$$
C_2
=
\begin{pmatrix}
0&-1&1&-1&0&1&0&1&2\\
0&1&-1&1&0&-1&0&-1&-2
\end{pmatrix}.
$$
All nonzero columns of this matrix are proportional to the vector
$$
\begin{pmatrix}
-1\\
1
\end{pmatrix}.
$$
Therefore
$$
\operatorname{rank} C_2=1.
$$
This agrees with the fact that
$$
\dim P_2^\Gamma
=
\dim A_2^\Gamma-\dim A_1^\Gamma
=
2-1=1.
$$

Thus, in the case \(n=4,\ k=2\), the space of primitive invariants
\(P_2^\Gamma\) is generated by the difference of the two orbital directions:
$$
-\mathcal R_{G_A}+\mathcal R_{G_D}.
$$
}
$\triangle$
\end{example}

\subsection{Primitive multiplicities, orbit numbers, and generating functions}

The transition from the Specht basis to orbital coordinates constructed
above has a direct consequence for the enumeration of graph orbits.

Summing over all subspaces, we obtain the formula for the number of all
non-isomorphic simple graphs on \(n\) vertices:
$$
g_n=\dim\mathcal A^\Gamma
=
\sum_{k=0}^{m}b_k.
$$
Passing to dimensions in formula \eqref{ab}, we have
$$
g_n
=
\sum_{j=0}^{\lfloor m/2\rfloor}
(m-2j+1)a_j, \quad a_j=\dim (P_j^\Gamma).
$$

The multiplicities \(a_j\) can be written in terms of characters. Since
$$
P_j^\Gamma
\cong
\left(\operatorname{Res}^{S_m}_{\Gamma}S^{(m-j,j)}\right)^\Gamma,
$$
we have
$$
a_j
=
\frac1{|\Gamma|}
\sum_{\gamma\in\Gamma}
\chi^{(m-j,j)}(\gamma).
$$

The last formula is a standard property of the Reynolds operator. If \(W\) is
a finite-dimensional \(\Gamma\)-module, then
$$
R_\Gamma=\frac1{|\Gamma|}\sum_{\gamma\in\Gamma}\gamma,
$$
is the projector onto \(W^\Gamma\). Therefore
$$
\dim W^\Gamma=\operatorname{tr}(R_\Gamma)
=
\frac1{|\Gamma|}\sum_{\gamma\in\Gamma}\chi_W(\gamma).
$$

Via the isomorphism \(\varphi:S_n\to\Gamma\), this can be rewritten as
$$
a_j
=
\frac1{n!}
\sum_{\sigma\in S_n}
\chi^{(m-j,j)}(\varphi(\sigma)),
$$
where \(\chi^{(m-j,j)}\) is the character of the Specht module
\(S^{(m-j,j)}\) of the group \(S_m\), and \(\varphi(\sigma)\) is the
permutation of edges induced by the permutation of vertices \(\sigma\).
Hence we obtain the following formula for the number of non-isomorphic simple
graphs on \(n\) vertices:
$$
g_n
=
\frac1{n!}
\sum_{\sigma\in S_n}
\sum_{j=0}^{\lfloor m/2\rfloor}
(m-2j+1)
\chi^{(m-j,j)}(\varphi(\sigma)).
$$

This formula is a representation-theoretic refinement of the classical
formula for the number of non-isomorphic graphs obtained by means of the
cycle index of the pair group:
$$
g_n
=
\frac1{n!}
\sum_{\sigma\in S_n}
2^{c(\varphi(\sigma))},
$$
see \cite{H-P}.
To see this connection, take \(\tau\in S_m\), and let \(N_r(\tau)\) denote
the number of \(r\)-element subsets of the set \([m]\) fixed by the
permutation \(\tau\). The character of the permutation module
\(M^{(m-r,r)}\) is equal to \(N_r(\tau)\).
By Young's rule for two-row permutation modules,
$$
M^{(m-r,r)}
\cong
\bigoplus_{j=0}^{\min(r,m-r)}
S^{(m-j,j)}.
$$
Therefore, for each \(r\), we have the equality of characters
$$
N_r(\tau)
=
\sum_{j=0}^{\min(r,m-r)}
\chi^{(m-j,j)}(\tau).
$$
Summing over all \(r=0,\dots,m\), we obtain
$$
\sum_{r=0}^{m}N_r(\tau)
=
\sum_{j=0}^{\lfloor m/2\rfloor}
(m-2j+1)\chi^{(m-j,j)}(\tau),
$$
because the fixed module \(S^{(m-j,j)}\) occurs in precisely those
permutation modules \(M^{(m-r,r)}\) for which
$$
j\le r\le m-j.
$$

The last sum
$$
\sum_{r=0}^{m}N_r(\tau)
$$
is the number of all subsets of the set \([m]\) fixed by the permutation
\(\tau\). If \(c(\tau)\) denotes the number of cycles of the permutation
\(\tau\), then such a subset must be a union of cycles of \(\tau\), and hence
$$
\sum_{r=0}^{m}N_r(\tau)=2^{c(\tau)}.
$$
Substituting \(\tau=\varphi(\sigma)\), we obtain the classical
Burnside--Polya formula.

Thus, Burnside's formula gives the external orbital count, whereas the
decomposition through \(a_j\) displays its internal \(\mathfrak{sl}_2\)- and
Specht-structure.

The same connection can be written at the level of generating functions.
Denote
$$
g_n(z)=\sum_{r=0}^{m} b_r z^r,
\qquad
p_n(z)=\sum_{j=0}^{\lfloor m/2\rfloor} a_jz^j.
$$
Since
$$
b_r=\sum_{j=0}^{\min(r,m-r)}a_j,
$$
each primitive contribution \(a_j\) appears in all subspaces
$$
j,j+1,\dots,m-j.
$$
Therefore
$$
g_n(z)
=
\sum_{j=0}^{\lfloor m/2\rfloor}
a_j(z^j+z^{j+1}+\cdots+z^{m-j}).
$$
Multiplying by \(1-z\), we obtain
$$
(1-z)g_n(z)
=
\sum_{j=0}^{\lfloor m/2\rfloor}
a_j(z^j-z^{m-j+1}).
$$
The first sum is equal to \(p_n(z)\), while the second is equal to
$$
\sum_{j=0}^{\lfloor m/2\rfloor}a_jz^{m-j+1}
=
z^{m+1}\sum_{j=0}^{\lfloor m/2\rfloor}a_jz^{-j}
=
z^{m+1}p_n(z^{-1}).
$$
Hence
$$
(1-z)g_n(z)
=
p_n(z)-z^{m+1}p_n(z^{-1}),
$$
and the generating function \(g_n(z)\) decomposes into contributions of full
\(\mathfrak{sl}_2\)-chains generated by primitive \(\Gamma\)-invariants.

On the other hand, for \(g_n(z)\) the determinant formula is known, see
\cite{BB2016}:
$$
g_n(z)
=
\frac1{n!}
\sum_{\sigma\in S_n}
\frac{\det(I_m-\varphi(\sigma) z^2)}
     {\det(I_m-\varphi(\sigma) z)}.
$$
Therefore \(p_n(z)\) can be recovered from it as the lower half of the
polynomial
$
(1-z)g_n(z):
$
$$
p_n(z)
=
\left[
(1-z)
\frac1{n!}
\sum_{\sigma\in S_n}
\frac{\det(I_m-\varphi(\sigma) z^2)}
     {\det(I_m-\varphi(\sigma) z)}
\right]_{\le \lfloor m/2\rfloor}.
$$
Here the symbol \([\cdot]_{\le \lfloor m/2\rfloor}\) means truncation of
terms of degree greater than \(\lfloor m/2\rfloor\). Thus, the determinant
formula for \(g_n(z)\) contains not only information about the number of graph
orbits in each subspace \(P_k\), but also information about the primitive
multiplicities \(a_j\), that is, about the Specht components which generate
these orbits.

\section{Conclusions and further work}

In this paper, the graph algebra was considered as an object of representation
theory. The starting point is a dual view of it: on the one hand, its monomial
basis is indexed by simple graphs on a fixed set of vertices; on the other
hand, it is the algebra of pseudo-Boolean functions on the Boolean cube of
edge configurations.

The main result of the paper is the construction and study of a natural
\(\mathfrak{sl}_2\)-structure on the graph algebra. The operators of adding
and deleting one edge define on \(\mathcal A\) the structure of an
\(\mathfrak{sl}_2\)-module, which makes it possible to decompose the algebra
into irreducible components. The primitive spaces arising in this
decomposition receive a representation-theoretic interpretation as two-row
Specht modules.

Separately, the compatibility of this structure with the action of the pair
group \(S_n^{(2)}\), induced by the renumbering of vertices, was established.
Since the \(\mathfrak{sl}_2\)-action commutes with the action of this group,
the space of graph invariants inherits the corresponding module structure.
This makes it possible to describe invariants not only through orbit sums of
graphs, but also through primitive components and their
\(\mathfrak{sl}_2\)-generations.

Schur--Weyl duality was used as a conceptual mechanism connecting the
\(\mathfrak{sl}_2\)-decomposition of the graph algebra with its permutation
structure on edges. In this approach, the classical orbital enumeration of
graphs receives a representation-theoretic refinement: the spaces whose
dimensions are counted by the Burnside--Polya formula decompose into natural
primitive contributions.

Further work may be directed toward a more detailed study of the
multiplicative structure of the invariant algebra, transition matrices between
Specht bases and orbit sums, and the construction of efficient algorithms for
computing primitive invariants in higher ranks.

\end{document}